\documentclass[reqno]{amsart}
\usepackage[sc]{mathpazo}

\usepackage{amsmath}
\usepackage{amsthm}
\usepackage{amsfonts}
\usepackage{graphicx,psfrag,epsfig}
\usepackage{color}

\usepackage{pdfsync}

\newtheorem{thm}{Theorem}[section]
\newtheorem{lem}[thm]{Lemma}

\newtheorem{prop}[thm]{Proposition}

\newtheorem{rem}[thm]{Remark}

\DeclareMathAlphabet{\mathpzc}{OT1}{pzc}{m}{it}

\numberwithin{equation}{section}

\newcommand{\bqn}{\begin{equation}}
\newcommand{\eqn}{\end{equation}}
\newcommand{\bqnn}{\begin{equation*}}
\newcommand{\eqnn}{\end{equation*}}

\newcommand{\R}{\mathbb{R}}

\newcommand{\N}{\mathbb{N}}
\newcommand{\C}{\mathbb{C}}

\newcommand{\ml}{\mathcal{L}}

\newcommand{\Om}{\Omega}

\newcommand{\ve}{\varepsilon}
\newcommand{\e}{\varepsilon}
\newcommand{\rd}{\mathrm{d}}

\newcommand{\ue}{u_\varepsilon}

\newcommand{\Phie}{\Phi_\varepsilon}

\date{\today}

\newcommand{\dhr}{\mathrel{\lhook\joinrel\relbar\kern-.8ex\joinrel\lhook\joinrel\rightarrow}} 

\title[A free boundary problem for MEMS with curvature]
{Dynamics of a free boundary problem with curvature modeling electrostatic MEMS}

\begin{document}
\author{Joachim Escher}
\address{Leibniz Universit\"at Hannover\\ Institut f\" ur Angewandte Mathematik \\ Welfengarten 1 \\ D--30167 Hannover\\ Germany}
\email{escher@ifam.uni-hannover.de}

\author{Philippe Lauren\c{c}ot}
\address{Institut de Math\'ematiques de Toulouse, CNRS UMR~5219, Universit\'e de Toulouse \\ F--31062 Toulouse Cedex 9, France}
\email{laurenco@math.univ-toulouse.fr}

\author{Christoph Walker}
\address{Leibniz Universit\"at Hannover\\ Institut f\" ur Angewandte Mathematik \\ Welfengarten 1 \\ D--30167 Hannover\\ Germany}
\email{walker@ifam.uni-hannover.de}

\begin{abstract}
The dynamics of a free boundary problem for electrostatically actuated microelectromechanical systems (MEMS) is investigated. The model couples the electric potential to the deformation of the membrane, the deflection of the membrane being caused by application of a voltage difference across the device. More precisely, the electrostatic potential is a harmonic function in the angular domain that is partly bounded by the deformable membrane. The gradient trace of the electric potential on this free boundary part acts as a source term in the governing equation for the membrane deformation. The main feature of the model considered herein is that, unlike most previous research, the small deformation assumption is dropped, and curvature for the deformation of the membrane is taken into account what leads to a quasilinear parabolic equation. The free boundary problem is shown to be well-posed, locally in time for arbitrary voltage values and globally in time for small voltages values. Furthermore, existence of asymptotically stable steady-state configurations is proved in case of small voltage values as well as non-existence of steady-states if the applied voltage difference is large. Finally, convergence of  solutions of the free boundary problem to the solutions of the well-established small aspect ratio model is shown.
\end{abstract}


\maketitle

\section{Introduction}

The focus of this paper is the analysis of a model describing the dynamics of an electrostatically actuated microelectromechanical system (MEMS) when the deformation of the devices are not assumed to be small.
 More precisely, consider an elastic plate held at  potential $V$ and suspended above a fixed ground plate held at zero potential. The potential difference between the two plates generates a Coulomb force and causes a deformation of the membrane, thereby converting electrostatic energy into mechanical energy, a feature used in the design of several MEMS-based devices such as micropumps or microswitches \cite{BernsteinPelesko}. An ubiquitous phenomenon observed in such devices is the so called ``pull-in'' instability: a threshold value of the applied voltage $V$ above which the elastic response of the membrane cannot balance the Coulomb force and the deformable membrane smashes into the fixed plate. Since this effect might either be useful or, in contrast, could damage the device, its understanding is of utmost practical importance and several mathematical models have been set up for its investigation \cite{EspositoGhoussoubGuo, LeusElata08, BernsteinPelesko}.

In the following subsection we give a brief description of an idealized device as depicted in Figure~\ref{fig1}, where the state of the device is characterized by the electrostatic potential in the region between the two plates and the deformation of the membrane which is not assumed to be small from the outset, cf. \cite{BrubakerPelesko_EJAM}.

\subsection{The Model}
To derive the model for electrostatic MEMS with curvature we proceed similarly to \cite{BrubakerPelesko_EJAM,EspositoGhoussoubGuo,LinYang}.
We consider a rectangular thin elastic membrane that is suspended above a rigid plate. The $(\hat{x},\hat{y},\hat{z})$-coordinate system is chosen such that the ground plate of dimension $[-L,L]\times [0,l]$ in $(\hat{x},\hat{y})$-direction is located at $\hat{z}=-H$, while the undeflected membrane with the same dimension $[-L,L]\times [0,l]$ in $(\hat{x},\hat{y})$-direction is located at $\hat{z}=0$. The membrane is held fixed along the edges in $\hat{y}$-direction while the edges in $\hat{x}$-direction are free. Assuming homogeneity in $\hat{y}$-direction, the membrane may thus be considered as an elastic strip and the $\hat{y}$-direction is omitted in the sequel. The mechanical deflection of the membrane is caused by a voltage difference that is applied across the device. The membrane is held at potential $V$ while the rigid plate is grounded. 
We denote the deflection of the membrane at position $\hat{x} $ and time $\hat{t}$ by $\hat{u}=\hat{u}(\hat{t},\hat{x}) >-H$ and the electrostatic potential at position $(\hat{x},\hat{z}) $ and time $\hat{t}$ by $\hat{\psi}=\hat{\psi}(\hat{t},\hat{x} ,\hat{z})$. We do not indicate the time variable $\hat{t}$  for  the time being.
The electrostatic potential $\hat{\psi}$ is harmonic, i.e.
\bqn\label{psihat}
\Delta\hat{\psi}=0 \quad\text{in}\quad \hat{\Omega}(\hat{u})
\eqn
and satisfies the boundary conditions
\bqn\label{psihatbc}
\hat{\psi}(\hat{x},-H)=0\ ,\quad \hat{\psi}(\hat{x},\hat{u}(\hat{x}))= V\ ,
 \qquad \hat{x}\in(-L,L)\ ,
\eqn
where
$$
\hat{\Omega}(\hat{u}):=\left\{(\hat{x},\hat{z})\,;\, -L<\hat{x}<L\,,\, -H<\hat{z}<\hat{u}(\hat{x})\right\}
$$
is the region between the ground plate and the membrane. The total energy of the system constitutes of the electric potential and the elastic energy and reads $$E(\hat{u})=E_e(\hat{u})+E_s(\hat{u})\ . $$ The {\it electrostatic energy} $E_e$ in dependence of the deflection $\hat{u}$  is given by
\bqnn 
E_e(\hat{u})=-\frac{\epsilon_0}{2}\int_{-L}^L\int_{-H}^{\hat{u}(\hat{x})}\vert\nabla\hat{\psi}(\hat{x},\hat{z})\vert^2\,\,\rd\hat{z}\,\rd \hat{x}
\eqnn 
with $\epsilon_0$ being the permittivity of free space while the {\it elastic energy} $E_s$ only retains the contribution due to stretching (in particular, bending is neglected) and is proportional
 to the tension $T$ and to the change of surface area of the membrane, i.e.
\bqnn 
E_s(\hat{u})=T\int_{-L}^L  \left(\sqrt{1+(\partial_{\hat{x}}\hat{u}(\hat{x}))^2}-1\right)\, \rd \hat{x}\ .
\eqnn
Introducing the dimensionless variables
$$
x=\frac{\hat{x}}{L}\ ,\quad z=\frac{\hat{z}}{H}\ ,\quad u=\frac{\hat{u}}{H}\ ,\quad \psi = \frac{\hat{\psi}}{V}
$$
and denoting  the aspect ratio of the device by $\ve=H/L$, we may write the total energy in these variables in the form
\bqn\label{E}
\begin{split}
E(u)&= TL\int_{-1}^1  \left(\sqrt{1+\ve^2(\partial_{x}u(x))^2}-1\right)\, \rd x\\
&\qquad -\frac{\epsilon_0V^2}{2\ve}\int_{\Omega(u)}\left(\ve^2\vert\partial_{x}\psi(x,z)\vert^2+\vert\partial_z\psi(x,z)\vert^2\right)\,\rd (x,z)\ ,
\end{split}
\eqn 
with
$$
\Omega(u) := \left\{ (x,z)\in (-1,1)\times (-1,\infty)\ :\ -1 < z < u(x) \right\}\ ,
$$
so that, formally, the corresponding Euler-Lagrange equations are
\bqn\label{EL}
\begin{split}
0=
\ve^2\,\partial_{x}\left(\frac{\partial_{x}u}{\sqrt{1+\ve^2(\partial_{x}u)^2}}\right)-\ve^2\,\lambda \left(\ve^2\vert\partial_{x}\psi(x,u(x))\vert^2+\vert\partial_z\psi(x,u(x))\vert^2\right)
\end{split}
\eqn
for $x\in (-1,1)$, where we have set
$$
\lambda=\frac{\epsilon_0V^2}{TL\ve^3}\ .
$$
We now take again time into account and derive the dynamics of the dimensionless deflection $u=u(\hat{t},x)$ by means of Newton's second law. Letting  $\rho$ and $\delta $ denote the mass density per unit volume of the membrane and the membrane thickness, respectively, the sum over all forces equals  $\rho \delta  \partial_{\hat{t}}^2 u$. The elastic and electrostatic forces, given by the right hand side of equation~\eqref{EL}, are combined with a damping force of the form $-a\partial_{\hat{t}} u$ being linearly proportional to the velocity. This yields
\bqnn
\begin{split}
\rho\, \delta \, \partial_{\hat{t}}^2 u +a\,\partial_{\hat{t}} u =\ve^2\, \partial_{x}\left(\frac{\partial_{x}u}{\sqrt{1+\ve^2(\partial_{x}u)^2}}\right)
-\ve^2\,\lambda \left(\ve^2\vert\partial_{x}\psi(x,u(x))\vert^2+\vert\partial_z\psi(x,u(x))\vert^2\right)\ .
\end{split}
\eqnn
Finally, scaling time based on the strength of damping according to $t=\hat{t}\ve^2/a$ and setting \mbox{$\gamma:=\frac{\sqrt{\rho \delta }\ve}{a}$}, we derive for the dimensionless deflection the evolution problem
\bqn\label{evol}
\begin{split}
\gamma^2\, \partial_{t}^2 u +\partial_{t} u\, =\, \partial_{x}\left(\frac{\partial_{x}u}{\sqrt{1+\ve^2(\partial_{x}u)^2}}\right)
-\,\lambda \left(\ve^2\vert\partial_{x}\psi(x,u(x))\vert^2+\vert\partial_z\psi(x,u(x))\vert^2\right)
\end{split}
\eqn
for $t>0$ and $x\in I:=(-1,1)$. Instead of considering this hyperbolic equation, however, we assume in this paper  that viscous or damping forces dominate over inertial forces, i.e. we assume that $\gamma \ll 1$ and thus neglect the second order time derivative term in~\eqref{evol}. The membrane displacement $u=u(t,x)\in (-1,\infty)$ then evolves according to
\begin{equation}\label{u}
\partial_t u - \partial_x\left(\frac{\partial_x u}{\sqrt{1+\varepsilon^2(\partial_x u)^2}}\right) = - \lambda\ \left( \varepsilon^2\ |\partial_x\psi(t,x,u(x))|^2 + |\partial_z\psi(t,x,u(x))|^2 \right)\ ,
\end{equation}
for $t>0$ and $x\in I$
with clamped boundary conditions
\begin{equation}\label{bcu}
u(t,\pm 1)=0\ ,\quad t>0\ ,
\end{equation}
and initial condition
\begin{equation}\label{ic}
u(0,x)=u^0(x)\ ,\quad x\in I\ .
\end{equation}
In dimensionless variables,  equations \eqref{psihat}-\eqref{psihatbc} read
\bqn\label{psi}
\ve^2 \,\partial_x^2\psi+\partial_z^2\psi=0 \ ,\quad (x,z)\in \Omega(u(t))\ ,\quad t>0\ ,
\eqn
subject to the boundary conditions (extended continuously to the lateral boundary)
\bqn\label{bcpsi}
\psi(t,x,z)=\frac{1+z}{1+u(t,x)}\ ,\quad (x,z)\in\partial\Omega(u(t))\ ,\quad t>0\ .
\eqn
In the following we shall focus our attention on  \eqref{u}-\eqref{bcpsi}, its situation being depicted in Figure~\ref{fig1}.

\begin{figure}
\centering\includegraphics[width=10cm]{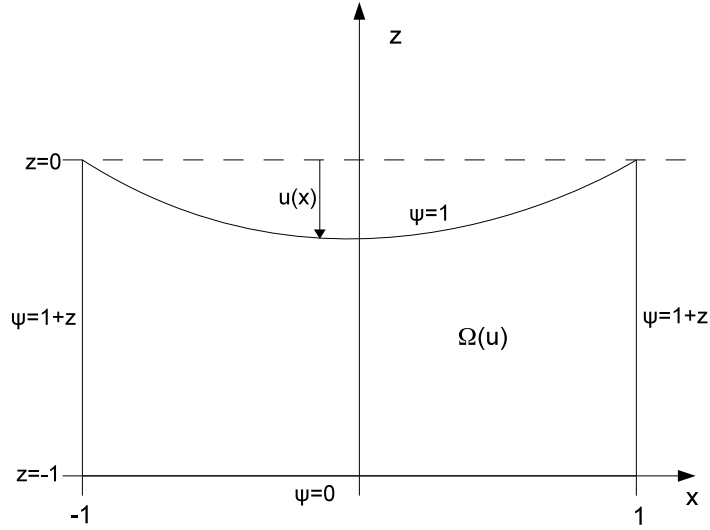}
\caption{\small Idealized electrostatic MEMS device.}\label{fig1}
\end{figure}

\subsection{Simplified Models}

Besides assuming that damping forces dominate over inertial forces and thus reducing equation \eqref{evol} to \eqref{u}, other simplifications of the model above have been considered as well in the literature. For instance,
restricting attention to small deformations of the membrane yields a linearized stretching term $\partial_x^2 u$ in \eqref{u}. The corresponding semilinear evolution problem with \eqref{u} being replaced by
\begin{equation}\label{UU}
\partial_t u - \partial_x^2 u = - \lambda\ \left( \varepsilon^2\ |\partial_x\psi(t,x,u(x))|^2 + |\partial_z\psi(t,x,u(x))|^2 \right)\ ,\quad x\in I\ ,\qquad t>0\ ,
\end{equation}
is investigated in \cite{ELW_CPAM}. It is shown therein that the problem  \eqref{bcu}-\eqref{UU} is well-posed locally in time. Moreover, solutions exist globally for small voltage values $\lambda$ while global existence is shown not to hold for high voltage values. It is also proven that, for small voltage values, there is an asymptotically stable steady-state solution. Finally, as the parameter $\ve$ approaches zero, the solutions are shown to converge toward the solutions of the so-called small aspect ratio model, see \eqref{z0} below. Indeed, letting $\ve=0$ (and applying a potential $V$ with $V^2 \sim \ve^3$ as suggested in \cite{BrubakerPelesko_EJAM}), one can solve  \eqref{psi}-\eqref{bcpsi} explicitly for the potential $\psi=\psi_0$, that is,
\bqn\label{psi0}
\psi_0(t,x,z)=\frac{1+z}{1+u_0(t,x)}\ ,\quad (t,x,z)\in [0,\infty)\times I\times(-1,0)\ ,
\eqn
and the displacement $u=u_0$ satisfies 
\begin{equation}\label{z0}
\begin{array}{rlll}
\partial_t u_0 - \partial_x^2 u_0 &\!\!\!=\!\!\!& - \displaystyle{\frac{\lambda}{(1+u_0)^2}}\,,\quad & x\in I\,, \quad t\in (0,\infty)\,, \\
u_0(t,\pm1) &\!\!\!=\!\!\!& 0\,, & t\in (0,\infty)\,, \\
u_0(0,x)&\!\!\!=\!\!\!& u^0(x)\,, & x\in I\, .
\end{array}
\end{equation}
In the limit $\ve\rightarrow 0$, the free boundary problem is thus reduced to the singular semilinear heat equation \eqref{z0} which has been studied thoroughly in recent years, see \cite{EspositoGhoussoubGuo} for a survey as well as e.g. \cite{FloresMercadoPeleskoSmyth_SIAP07,GhoussoubGuoNoDEA08,GuoJDE08,GuoJDE08_II,GuoHuWang09,GuoKavallaris,Hui11,LinYang,PeleskoSIAP02}. It is noteworthy to remark that the picture regarding pull-in voltage for the small aspect ratio model \eqref{z0} is rather complete. \\

Let us point out that \cite{ELW_CPAM} is apparently the first mathematical analysis of the parabolic free boundary problem  \eqref{bcu}-\eqref{UU} while the corresponding elliptic (i.e. steady-state) free boundary problem is investigated in \cite{LaurencotWalker_ARMA}. Moreover, we shall emphasize that the inclusion of non-small deformations is a feature of great physical relevance and, even though the results presented herein are reminiscent of the ones in  \cite{ELW_CPAM},  the quasilinear structure of \eqref{u} is by no means a trivial mathematical extension of \eqref{UU}.

\subsection{Main Results}

To state our findings on \eqref{u}-\eqref{bcpsi}, we introduce the spaces
\begin{equation}
W_{q,D}^{2\alpha}(I):=\left\{ 
\begin{array}{lcl} 
\{u\in  W_{q}^{2\alpha}(I)\,;\, u(\pm 1)=0\} & \text{ for } & 2\alpha \in \left( {\frac{1}{q}},2 \right]\ ,\\
& & \\
W_{q}^{2\alpha}(I) & \text{ for } & 0\le 2\alpha < {\frac{1}{q}}\ .
\end{array}
\right. \label{wap}
\end{equation}

We shall prove the following result regarding local and global existence of solutions:

\begin{thm}[{\bf Well-Posedness}]\label{A}
Let {  $q\in (2,\infty)$}, $\varepsilon>0$, and consider an initial value $u^0\in  W_{q,D}^2(I)$ such that $u^0(x)>-1$ for $x\in I$. Then, the following are true:

\begin{itemize}

\item[(i)] For each voltage value $\lambda>0$, there is a unique maximal solution $(u,\psi)$ to \eqref{u}-\eqref{bcpsi} on the maximal interval of existence $[0,T_m^\varepsilon)$ in the sense that
$$u\in C^1\big([0,T_m^\varepsilon),L_q(I)\big)\cap C\big([0,T_m^\varepsilon), W_{q,D}^2(I)\big)
$$
satisfies \eqref{u}-\eqref{ic} together with
$$
u(t,x)>-1\ ,\quad (t,x)\in [0,T_m^\varepsilon)\times I\ , 
$$ 
and $\psi(t)\in W^2_{  2}\big(\Omega(u(t))\big)$ solves \eqref{psi}-\eqref{bcpsi} on $\Omega(u(t))$ for each $t\in [0,T_m^\varepsilon)$. 

\item[(ii)] If for each $\tau>0$ there is $\kappa(\tau)\in (0,1)$ such that $u(t)\ge -1+\kappa(\tau)$ and $\|u(t)\|_{W_q^2(I)}\le \kappa(\tau)^{-1}$ for $t\in [0,T_m^\varepsilon)\cap [0,\tau]$, then the solution exists globally, that is, $T_m^\varepsilon=\infty$.

\item[(iii)] If $u^0(x)\le 0$ for $x\in I$, then $u(t,x)\le 0$ for $(t,x)\in [0,T_m^\varepsilon)\times I$. If $u^0=u^0(x)$ is even with respect to $x\in I$, then, {for all $t\in [0,T_m^\varepsilon)$}, $u=u(t,x)$ and $\psi=\psi(t,x,z)$ are even with respect to $x\in I$ as well.

\item[(iv)] Given $\kappa\in (0,1)$, there are $\lambda_*(\kappa)>0$ and $r(\kappa)>0$ such that $T_m^\ve=\infty$ {  with $u(t,x)\ge -1+\kappa$ for $(t,x)\in [0,\infty)\times I$ provided that $\lambda\in (0,\lambda_*(\kappa))$ and $\|u^0\|_{W_q^2(I)}\le r(\kappa)$. In that case, $u$ enjoys the following additional regularity properties: 
$$
u\in BUC^\rho([0,\infty),W_{q,D}^{2-\rho}(I))\cap L_\infty([0,\infty),W_{q,D}^{2}(I))
$$ 
for some $\rho>0$ small.}
\end{itemize}
\end{thm}

Note that part (iv) of Theorem~\ref{A} provides uniform estimates on $u$ in the $W_{q}^2(I)$-norm and ensures that $u$ never touches down on -1, not even in infinite time. In contrast to the semilinear case considered in \cite{ELW_CPAM}, the global existence result for the quasilinear equation \eqref{u} requires initially a small deformation, see also  Remark~\ref{Remm} below. The proof of Theorem~\ref{A} is the content of Section~\ref{Sec3}. It is based on interpreting \eqref{u} as a abstract quasilinear Cauchy problem which allows us to employ the powerful theory of evolution operators developed in \cite{LQPP}. Let us emphasize at this point that the regularity properties of the right-hand side of equation \eqref{u} established in \cite{ELW_CPAM} are not sufficient to handle the quasilinear character of the curvature operator and we consequently have to derive Lipschitz properties of the right-hand side of \eqref{u} in weaker topologies than in \cite{ELW_CPAM}. This is the purpose of Section~\ref{Sec2}.

\medskip

Regarding existence and asymptotic stability of steady-state solutions to \eqref{u}-\eqref{bcpsi} we have a similar result as in \cite[Thm.~1]{LaurencotWalker_ARMA} and \cite[Thm.1.3]{ELW_CPAM}.

\begin{thm}[{\bf Asymptotically Stable Steady-State Solutions}]\label{TStable}
Let $q\in (2,\infty)$ and $\ve>0$. 
\begin{itemize}
\item[(i)] Let $\kappa\in (0,1)$. There are $\delta=\delta(\kappa)>0$ and an analytic function $$[\lambda\mapsto U_\lambda]:[0,\delta)\rightarrow W_{q,D}^2(I)$$ such that $(U_\lambda,\Psi_\lambda)$ is the unique steady-state to
\eqref{u}-\eqref{bcpsi} satisfying \mbox{$\|U_\lambda\|_{W_{q,D}^2(I)}\le 1/\kappa$} with $U_\lambda\ge -1+\kappa$ and $\Psi_\lambda\in W_2^2(\Omega(U_\lambda))$ when $\lambda\in (0,\delta)$. The steady-state possesses the additional regularity
\begin{equation}
\begin{split}
U_\lambda &\in C^{2+\alpha}\big([-1,1]\big)\ ,\\ 
\Psi_\lambda &\in W_2^2\big(\Omega(U_\lambda)\big) \cap C\big(\overline{\Omega(U_\lambda)}\big)\cap C^{2+\alpha}\big(\Omega(U_\lambda)\cup\Gamma(U_\lambda)\big)\ , 
\end{split} \label{new1}
\end{equation}
where $\alpha\in [0,1)$ is arbitrary and  $\Gamma(U_\lambda)$ denotes the boundary of $\Omega(U_\lambda)$ without corners.
Moreover, $U_\lambda$ is negative, convex, and even with $U_0=0$ and $\Psi_\lambda=\Psi_\lambda(x,z)$ is even with respect to $x\in I$.

\item[(ii)] Let $\lambda\in (0,\delta)$. There are $\omega_0,m,R>0$ such that for each initial value \mbox{$u^0\in W_{q,D}^2(I)$} with $\|u^0-U_\lambda\|_{W_{q,D}^2} <m$,  \eqref{u}-\eqref{bcpsi} has a unique global solution $(u,\psi)$ with
$$
u\in C^1\big([0,\infty),L_q(I)\big)\cap C\big([0,\infty), W_{q,D}^2(I)\big)\ ,\qquad \psi(t)\in W_2^2\big(\Omega(u(t))\big)\ ,\quad t\ge 0\ ,
$$
and
$$
u(t,x)>-1\ ,\quad (t,x)\in [0,\infty)\times I\ .
$$
Moreover,
\bqn\label{est}
\|u(t)-U_\lambda\|_{W_{q,D}^2(I)}+\|\partial_t u(t)\|_{L_{q}(I)} \le R e^{-\omega_0 t} \|u^0-U_\lambda\|_{W_{q,D}^2(I)}\ ,\quad t\ge 0\ .
\eqn
\end{itemize}
\end{thm}

Part (ii) of Theorem~\ref{TStable} shows local exponential stability of the steady-states derived in part~(i). We also point out that the potential $\psi$ converges exponentially to $\Psi_\lambda$ in the $W_2^2$-norm as $t\rightarrow \infty$, see Remark~\ref{RR} for a precise statement. The proof of Theorem~\ref{TStable} is given in Section~\ref{SectTStable} and relies on the Implicit Function Theorem for part (i) and the Principle of Linearized Stability for part (ii).\\

Clearly, Theorem~\ref{TStable} is just a local result with respect to $\lambda$ values. However, we next show that there is an upper threshold for $\lambda$ above which no steady-state solution exists. This is expected on physical grounds and is related to the ``pull-in'' instability already mentioned in the introduction. 

\begin{thm}[{\bf Non-Existence of Steady-State Solutions}]\label{NoSS}
Let $q\in (2,\infty)$ and $\ve>0$. 
There is $\bar{\lambda}(\ve)>0$ such that, if $\lambda\ge \bar{\lambda}(\ve)$, then there is no steady-state solution $(u,\psi)$ to \eqref{u}-\eqref{bcpsi} such that $u\in W_{q,D}^2(I)$, $\psi\in W_2^2(\Omega(u))$, and $u(x)>-1$ for $x\in I$. In addition, $\bar{\lambda}(\ve)\rightarrow 2$ as $\ve\rightarrow 0$.
\end{thm}

Similar results have already been obtained for related models, including the small aspect ratio model \cite{BrubakerPelesko_EJAM,EspositoGhoussoubGuo} and for the stationary free boundary problem corresponding to \eqref{bcu}-\eqref{UU}, see \cite{LaurencotWalker_ARMA}.
The proof of Theorem~\ref{NoSS} relies on a lower bound on $\partial_z\psi (x,u(x))$ established in the latter paper and is given in Section~\ref{SectNoSS}.\\

The final issue we address is the connection between the free boundary problem \eqref{u}-\eqref{bcpsi} and the small aspect ratio limit \eqref{z0}. More precisely, we show the following convergence result:

\begin{thm}[{\bf Small Aspect Ratio Limit}]\label{Bq}
Let $\lambda>0$, $q\in (2,\infty)$, and let $u^0\in W_{q,D}^2(I)$  with $-1<u^0(x)\le 0$ for $x\in I$. For $\varepsilon>0$ we denote the unique solution to \eqref{u}-\eqref{bcpsi} on the maximal interval of existence $[0,T_m^\varepsilon)$ by $(u_\varepsilon,\psi_\varepsilon)$. There are $\tau>0$, $\ve_0>0$, and $\kappa_0\in (0,1)$ depending only on $q$ and $u^0$ such that $\ T_m^\varepsilon\ge \tau$ with $u_\ve(t)\ge -1+\kappa_0$ and $\|u_\ve (t)\|_{W_{q}^2(I)}\le \kappa_0^{-1}$ for all $(t,\ve)\in [0,\tau]\times (0,\ve_0)$. 
Moreover, as $\ve\rightarrow 0$,
$$ 
u_{\varepsilon}\longrightarrow u_0\quad \text{in}\quad C^{1-\theta}\big([0,\tau], W_q^{2\theta}(I)\big)\ ,\quad 0<\theta<1\ ,
$$
and
\begin{equation}\label{z00} 
\psi_{\varepsilon}(t)\mathbf{1}_{\Omega(u_{\varepsilon}(t))}\longrightarrow \psi_{0}(t)\mathbf{1}_{\Omega(u_{0}(t))}\quad \text{in}\quad L_2\big(I\times (-1,0)\big)\ ,\quad t\in [0,\tau]\ ,
\end{equation}
where 
$$
u_0\in C^1\big([0,\tau],L_q(I)\big)\cap C\big([0,\tau],W_{q,D}^2(I)\big)
$$
is the unique solution to the small aspect ratio equation \eqref{z0} and  $\psi_0$ is the potential given in \eqref{psi0}. 
Furthermore, there is $\Lambda(u^0)>0$ such that the results above hold true for each  $\tau>0$ provided that $\lambda\in (0,\Lambda(u^0))$.
\end{thm}

The proof is given in Section~\ref{Sec6}. A similar result has been established in \cite[Thm.~2]{LaurencotWalker_ARMA} for the stationary problem and in \cite[Thm.1.4]{ELW_CPAM} for the semilinear parabolic version \eqref{UU}. As in the latter paper, the crucial step is to derive the $\ve$-independent lower bound $\tau>0$ on $T_m^\ve$, which is not guaranteed by the analysis leading to Theorem~\ref{A}. The proof of Theorem~\ref{Bq} uses several properties of \eqref{psi}-\eqref{bcpsi} with respect to the $\ve$-dependence shown in \cite{ELW_CPAM}.

\section{On the Elliptic Equation \eqref{psi}-\eqref{bcpsi} }\label{Sec2} 

We shall first derive properties of solutions to the elliptic equation \eqref{psi}-\eqref{bcpsi} in dependence of a given (free) boundary. To do so, we transform the free boundary problem \eqref{psi}-\eqref{bcpsi} to the fixed rectangle $\Omega:=I\times (0,1)$. More precisely, let $q\in (2,\infty)$ be fixed and consider an arbitrary function $v\in W_{q,D}^2(I)$ taking values in $(-1,\infty)$. We then define a diffeomorphism \mbox{$T_v:\overline{\Omega(v)}\rightarrow \overline{\Omega}$} by setting
\begin{equation}\label{Tu}
T_v(x,z):=\left(x,\frac{1+z}{1+v(x)}\right)\ ,\quad (x,z)\in \overline{\Omega(v)}\ ,
\end{equation}
with $\Omega(v) := \left\{ (x,z)\in I\times (-1,\infty)\ ;\-1 < z < v(x) \right\}$. Clearly, its inverse is
\begin{equation}\label{Tuu}
T_v^{-1}(x,\eta)=\big(x,(1+v(x))\eta-1\big)\ ,\quad (x,\eta)\in \overline{\Omega}\ ,
\end{equation}
and the Laplace operator from \eqref{psi} is transformed to the $v$-dependent differential operator 
\begin{equation*}
\begin{split}
\mathcal{L}_v w\, :=\, & \e^2\ \partial_x^2 w - 2\e^2\ \eta\ \frac{\partial_x v(x)}{1+v(x)}\ \partial_x\partial_\eta w
+ \frac{1+\e^2\eta^2(\partial_x v(x))^2}{(1+v(x))^2}\ \partial_\eta^2 w\\
& + \e^2\ \eta\ \left[ 2\ \left(\frac{\partial_x v(x)}{1+v(x)} \right)^2 - \frac{\partial_x^2 v(x)}{1+v(x)} \right]\ \partial_\eta w\ .
\end{split}
\end{equation*}
An alternative formulation of the boundary value problem  \eqref{psi}-\eqref{bcpsi} is then
\begin{eqnarray}
\big(\mathcal{L}_{u(t)}\phi\big) (t,x,\eta)\!\!\!&=0\ ,&(x,\eta)\in\Omega\ , \quad \ t>0\ ,\label{23}\\
\phi(t,x,\eta)\!\!\!&=\eta\ , &(x,\eta)\in \partial\Omega\ , \quad t>0\ ,\label{24}
\end{eqnarray}
for $\phi=\psi\circ T_{u(t)}^{-1}$.
With this notation, the quasilinear evolution equation \eqref{u} for $u$ becomes
\begin{align}
\partial_t u -\partial_x\left(\frac{\partial_x u}{\sqrt{1+\varepsilon^2(\partial_x u)^2}}\right) &= {-\lambda}\ \left[ \frac{1+\e^2(\partial_x u)^2}{(1+u)^2} \right]\ \vert\partial_\eta\phi(\cdot,1)\vert^2\ ,\quad x\in I\ ,\quad t>0\ , \label{33}
\end{align}
where we have used $\partial_x\phi(t,x,1)=0$ for $x\in I$ and $t>0$ due to $\phi(t,x,1)=1$ by \eqref{24}. The investigation of the dynamics of \eqref{33} involves the properties of its nonlinear right hand side as well as the properties of the quasilinear curvature term. We shall see that these two features of \eqref{33} are somewhat opposite as the treatment of the former requires a functional analytic setting in $W_q^2(I)$ to handle the second order terms of $\ml_{u(t)}$ in \eqref{23}, while a slightly weaker setting has to be chosen to guarantee H\"older continuity of $u$ with respect to time which is required in quasilinear evolution equations (see Remark~\ref{Remm} for further details). To account for these features of \eqref{33} we have to refine the Lipschitz property of
the right-hand side of \eqref{33} derived in \cite{ELW_CPAM} as stated in \eqref{gLip} below. \\

Defining for $\kappa\in (0,1)$ the open subset
$$
S_q(\kappa):=\left\{u\in W_{q,D}^2(I)\,;\, \|u\|_{W_{q,D}^2(I)}< 1/\kappa \;\;\text{ and }\;\; -1+\kappa< u(x) \text{ for } x\in I \right\}
$$
of $W_{q,D}^2(I)$ (defined in \eqref{wap}) with closure 
$$
\overline{S}_q(\kappa)=\left\{u\in W_{q,D}^2(I)\,;\, \|u\|_{W_{q,D}^2(I)}\le 1/\kappa \;\;\text{ and }\;\; -1+\kappa\le u(x) \text{ for } x\in I \right\}\ ,
$$ 
the crucial properties of the nonlinear right-hand side of \eqref{33} are collected in the following proposition:

\begin{prop}\label{L1}
Let $q\in (2,\infty)$, $\kappa\in (0,1)$, and $\varepsilon>0$. For each $v\in \overline{S}_q(\kappa)$ there is a unique solution $\phi_v\in W_2^2(\Omega)$ to 
\begin{eqnarray}
\big(\mathcal{L}_v \phi_v\big) (x,\eta)\!\!\!&=0\ ,&(x,\eta)\in\Omega\ ,\label{230}\\
\phi_v(x,\eta)\!\!\!&=\eta\ , &(x,\eta)\in \partial\Omega\ .\label{240}
\end{eqnarray}
If $\tilde{v}$ is defined by $\tilde{v}(x):=v(-x)$ for $x\in I$, then $\phi_{\tilde{v}}(x,\eta)=\phi_v(-x,\eta)$ for $(x,\eta)\in\Omega$. 
Moreover, for $2\sigma\in [0,1/2)$, the mapping
$$
g_\ve: S_q(\kappa)\longrightarrow W_{2,D}^{2\sigma}(I)\ ,\quad v\longmapsto \frac{1+\e^2(\partial_x v)^2}{(1+v)^2}\  \vert\partial_\eta\phi_v(\cdot,1)\vert^2
$$
is analytic and bounded with $g_\ve(0)=1$. Finally,  {  if $\xi\in [0,1/2)$ and $\nu\in [0,(1-2\xi)/2)$, then}  there exists a constant $c_1(\kappa,\ve)>0$ such that
\bqn\label{gLip}
\|g_\ve(v)-g_\ve(w)\|_{W_{2,D}^\nu(I)}\le c_1(\kappa,\ve)\,\|v-w\|_{W_{q,D}^{2-\xi}(I)}\ ,\quad v,w\in \overline{S}_q(\kappa)\ . 
\eqn
\end{prop}

According to \cite[Prop.~2.1]{ELW_CPAM} we actually only have to prove \eqref{gLip}. Notice that this global Lipschitz property is in the weaker topology of $W_{q,D}^{2-\xi}(I)$ instead of $W_{q,D}^{2}(I)$ and improves \cite[Prop.~2.1]{ELW_CPAM} where it was established for $\xi=0$. The property \eqref{gLip} will be a consequence of a sequence of lemmas. For the remainder of this section we fix $\ve>0$, $\kappa\in (0,1)$, and $q\in (2,\infty)$.

In the following, if $\alpha>1/2$ we let $W_{2,D}^\alpha(\Omega)$ denote the subspace of elements in $W_{2}^\alpha(\Omega)$ whose boundary trace is zero, and if $0\le \alpha<1/2$ we set $W_{2,D}^\alpha(\Omega):=W_{2}^\alpha(\Omega)$. We equip $W_{2,D}^1(\Omega)$ with the norm 
$$
\|\Phi\|_{W_{2,D}^1(\Omega)}:=\left(\|\partial_x\Phi\|_{L_2(\Omega)}^2 +\|\partial_\eta\Phi\|_{L_2(\Omega)}^2\right)^{1/2}\ ,
$$
and introduce the notation $$
W_{2,D}^{-\theta}(\Omega):= (W_{2,D}^{\theta}(\Omega))'\ ,\quad 0\le \theta\le 1  \ .
$$

\begin{lem}\label{L1t}
For each $v\in \overline{S}_q(\kappa)$ and $F\in W_{2,D}^{-1}(\Omega)$ there is a unique solution $\Phi\in W_{2,D}^1(\Omega)$ to the boundary value problem
\begin{eqnarray}
-\mathcal{L}_v \Phi  \!\!\!&=F&\text{in } \Omega\ ,\label{230a}\\
\Phi\!\!\!&=0 &\text{on } \partial\Omega\ ,\label{240a}
\end{eqnarray}
and there is a constant  $c_2(\kappa,\ve)>0$ such that
\bqn\label{2.3}
\|\Phi\|_{W_{2,D}^{1}(\Omega)}\le c_2(\kappa,\ve)\, \|F\|_{W_{2,D}^{-1}(\Omega)}\ .
\eqn
Furthermore, if $F\in L_2(\Omega)$, then $\Phi\in W_{2,D}^2(\Omega)$ and 
\bqn\label{2.4}
\|\Phi\|_{W_{2,D}^{2}(\Omega)}\le c_2(\kappa,\ve)\, \|F\|_{L_2(\Omega)}\ .
\eqn
\end{lem}

\begin{proof}[{\bf Proof}]
According to \cite[Def.~1.3.2.3, Eq. (1,3,2,3)]{Grisvard}, we may write any $F\in W_{2,D}^{-1}(\Omega)$ in the form \mbox{$F=f_0+\partial_x f_1+\partial_\eta f_2$} with $(f_0,f_1,f_2)\in L_2(\Omega)^3$. Consequently, \cite[Thm.~8.3]{GilbargTrudinger} ensures that the boundary value problem \eqref{230a}-\eqref{240a} has a unique solution $\Phi\in W_{2,D}^1(\Omega)$. Furthermore, taking $\Phi$ as a test function in the weak formulation of \eqref{230a}-\eqref{240a} gives { 
\begin{equation*}
\begin{split}
\langle F,\Phi\rangle =& \int_\Omega \left[ \ve^2(\partial_x\Phi)^2 - 2\,\ve^2\,\eta\,\frac{\partial_x v}{1+v}\, \partial_x\Phi\,\partial_\eta \Phi - \ve^2\, \eta\, \partial_x\left( \frac{\partial_x v}{1+v} \right)\, \Phi\, \partial_\eta\Phi - \ve^2\, \frac{\partial_x v}{1+v}\, \Phi\, \partial_x\Phi \right] \rd(x,\eta) \\
&+ \int_\Omega \left[ \frac{1+\ve^2\eta^2(\partial_xv)^2}{(1+v)^2}\, (\partial_\eta \Phi)^2 + 2\, \ve^2\, \eta\, \left( \frac{\partial_x v}{1+v} \right)^2\, \Phi\, \partial_\eta\Phi \right]\ \rd (x,\eta)\\
&+\int_\Omega \ve^2\, \eta\, \left[ \partial_x \left( \frac{\partial_x v}{1+v} \right) - \left( \frac{\partial_x v}{1+v} \right)^2 \right]\,\Phi\, \partial_\eta \Phi\ \rd (x,\eta)\\
=&\int_\Omega\left[\ve^2(\partial_x\Phi)^2-2\,\ve^2\,\eta\,\frac{\partial_x v}{1+v}\, \partial_x\Phi\,\partial_\eta \Phi +\frac{1+\ve^2\eta^2(\partial_xv)^2}{(1+v)^2}\, (\partial_\eta \Phi)^2\right]\ \rd (x,\eta)\\
&-\int_\Omega \ve^2\, \frac{\partial_x v}{1+v}\,\left[\partial_x\Phi-\eta\,\frac{\partial_x v}{1+v}\,\partial_\eta\Phi \right]\, \Phi\ \rd (x,\eta)
\end{split}
\end{equation*}
and} thus
\begin{equation}
\begin{split}\label{2.5}
\int_\Omega &\left[\ve^2\left(\partial_x\Phi\,-\,\eta\, \frac{\partial_x v}{1+v}\,\partial_\eta\Phi\right)^2\, +\,\left(\frac{\partial_\eta\Phi}{1+v}\right)^2\right]\, \rd (x,\eta)\\
& \le \ve^2 \left\|\frac{\partial_x v}{1+v}\right\|_{L_\infty(I)}\, \left\|\partial_x\Phi-\eta\,\frac{\partial_x v}{1+v}\,\partial_\eta\Phi\right\|_{L_2(\Omega)}\, \|\Phi\|_{L_2(\Omega)}\, +\, \|F\|_{W_{2,D}^{-1}(\Omega)}\, \|\Phi\|_{W_{2,D}^1(\Omega)}\ .
\end{split}
\end{equation}
Note then that by definition of $\overline{S}_q(\kappa)$ and Sobolev's embedding,
there is $c_3>0$ such that
\bqn\label{2.6}
1+v(x)\ge \kappa\ ,\quad x\in I\ ,\qquad \|v\|_{C^1([-1,1])}\le \frac{c_3}{\kappa}
\eqn
for all $v\in \overline{S}_q(\kappa)$. Also, if $\zeta=(\zeta_1,\zeta_2)\in\R^2$, Young's inequality ensures that, for $(x,\eta)\in \Omega$,
\begin{equation*}
\begin{split}
\ve^2\,\zeta_1^2\,&\le 2\,\ve^2\,\left(\zeta_1-\eta\,\frac{\partial_x v(x)}{1+v(x)}\,\zeta_2\right)^2+2\,\ve^2\eta^2\, \left(\frac{\partial_x v(x)}{1+v(x)}\right)^2\,\zeta_2^2\\
&\le
2\,\left(1+2\,\ve^2\,\|\partial_x v\|_\infty^2\right)\,\left(\ve^2\,\left(\zeta_1-\eta\,\frac{\partial_x v(x)}{1+v(x)}\,\zeta_2\right)^2+\frac{1}{2} \left(\frac{\zeta_2}{1+v(x)}\right)^2\right)\ .
\end{split}
\end{equation*}
Therefore, introducing
$$
\nu(\kappa,\ve):=\frac{1}{2}\,\min\left\{\frac{\ve^2\, \kappa^2}{\kappa^2 +2\ve^2 c_3^2}\,,\,\frac{\kappa^2}{(\kappa + c_3)^2}\right\}\ ,
$$
we infer from \eqref{2.6} that
\bqn\label{2.7}
\nu(\kappa,\ve) \big(\zeta_1^2+\zeta_2^2\big)\le \ve^2\left(\zeta_1-\eta\,\frac{\partial_x v(x)}{1+v(x)}\,\zeta_2\right)^2+ \left(\frac{\zeta_2}{1+v(x)}\right)^2\ .
\eqn
Consequently, \eqref{2.5}, \eqref{2.6},  and \eqref{2.7} give
\begin{equation*}
\begin{split}
&\ve^2\, \left\|\partial_x\Phi\,-\,\eta\, \frac{\partial_x v}{1+v}\,\partial_\eta\Phi\right\|_{L_2(\Omega)}^2\, +\,\left\|\frac{\partial_\eta\Phi}{1+v}\right\|_{L_2(\Omega)}^2\\
&\qquad\qquad\le \ve^2\frac{c_3}{\kappa^2} \left\|\partial_x\Phi\,-\,\eta\, \frac{\partial_x v}{1+v}\,\partial_\eta\Phi\right\|_{L_2(\Omega)}\| \Phi\|_{L_2(\Omega)}\\
&\quad \qquad\qquad + \frac{\|F\|_{W_{2,D}^{-1}(\Omega)}}{\sqrt{\nu(\kappa,\ve)}}\,\left(\ve^2\, \left\|\partial_x\Phi\,-\,\eta\, \frac{\partial_x v}{1+v}\,\partial_\eta\Phi\right\|_{L_2(\Omega)}^2+\left\|\frac{\partial_\eta \Phi}{1+v}\right\|_{L_2(\Omega)}^2\right)^{1/2}\ ,
\end{split}
\end{equation*}
whence, using again \eqref{2.7},
\bqn\label{2.8}
\|\Phi\|_{W_{2,D}^{1}(\Omega)}\,\le\, \frac{\ve}{\sqrt{\nu(\kappa,\ve)}}\,\frac{c_3}{\kappa^2}\,\|\Phi\|_{L_2(\Omega)}+\frac{\|F\|_{W_{2,D}^{-1}(\Omega)}}{\nu(\kappa,\ve)}\ .
\eqn
We now proceed as in \cite[Lem.~9.17]{GilbargTrudinger} and argue by contradiction to show \eqref{2.3}  (see also \cite[Lem.~6.2]{ELW_CPAM} for a similar argument in a slightly different functional setting). The last statement of Lemma~2.2 is proved in \cite[Lem.~6.2]{ELW_CPAM}.
\end{proof}

Now, introducing
\begin{equation}\label{f}
f_v(x,\eta):=\mathcal{L}_v\eta =\e^2\eta\left[2\left(\frac{\partial_x v(x)}{1+v(x)}\right)^2-\frac{\partial_x^2 v(x)}{1+v(x)}\right]\ ,\quad (x,\eta)\in\Omega\ ,
\end{equation}
for $v\in\overline{S}_q(\kappa)$ given, 
we readily deduce that $f_v\in L_2(\Omega)$ with
\begin{equation}
\|f_v\|_{L_2(\Omega)} \le c_4(\kappa,\ve)\ . \label{spip}
\end{equation}
Consequently, Lemma~\ref{L1t} provides a unique solution $\Phi_v\in W_{2,D}^2(\Omega)$ to
\begin{eqnarray*}
-\big(\mathcal{L}_v \Phi_v\big)  (x,\eta)\!\!\!&=f_v\ ,&(x,\eta)\in\Omega\ ,\nonumber\\
\Phi_v(x,\eta)\!\!\!&=0\ , &(x,\eta)\in \partial\Omega\ . 
\end{eqnarray*}
Clearly, defining 
\begin{equation}
\phi_v(x,\eta) := \Phi_v(x,\eta)+\eta\ , \qquad (x,\eta)\in\bar{\Omega}\ , \label{phiv}
\end{equation} 
gives then the unique solution $\phi_v\in W_2^2(\Omega)$ to \eqref{230}-\eqref{240}. To prove a Lipschitz dependence of $\phi_v$ on $v\in\overline{S}_q(\kappa)$, we introduce  a bounded linear operator $$\mathcal{A}(v)\in\ml\big(W_{2,D}^1(\Omega),W_{2,D}^{-1}(\Omega)\big)\cap \ml\big(W_{2,D}^2(\Omega),L_2(\Omega)\big)
$$
by setting
$$
\mathcal{A}(v)\Phi:=-\ml_v \Phi\ ,\quad \Phi\in W_{2,D}^1(\Omega)\ .
$$
Note that $\mathcal{A}(v)$ is  invertible according to Lemma~\ref{L1t} and that $\Phi_v=\mathcal{A}(v)^{-1} f_v$. For the inverse $\mathcal{A}(v)^{-1}$ we have:

\begin{lem}\label{L1b}
Given $\theta\in [0,1]\setminus\{1/2\}$, there is a constant $c_5(\kappa,\ve)>0$ such that
$$
\|\mathcal{A}(v)^{-1}\|_{\ml(W_{2,D}^{\theta-1}(\Omega),W_{2,D}^{\theta+1}(\Omega))}\,\le\, c_5(\kappa,\ve)\ ,\quad v\in\overline{S}_q(\kappa)\ .
$$
\end{lem}

\begin{proof}[{\bf Proof}]
Due to Lemma~\ref{L1t}, $\mathcal{A}(v)^{-1}$ belongs for each $v\in\overline{S}_q(\kappa)$ to both $\ml(W_{2,D}^{-1}(\Omega),W_{2,D}^1(\Omega))$ and $\ml(L_2(\Omega),W_{2,D}^2(\Omega))$ with
$$
\|\mathcal{A}(v)^{-1}\|_{\ml(W_{2,D}^{-1}(\Omega),W_{2,D}^1(\Omega))}\, +\, \|\mathcal{A}(v)^{-1}\|_{\ml(L_2(\Omega),W_{2,D}^2(\Omega))}\, \le\, 2\,c_2(\kappa,\ve) \ .
$$
Hence, using complex interpolation, we derive for $\theta\in [0,1]$,
$$
\|\mathcal{A}(v)^{-1}\|_{\ml\big([W_{2,D}^{-1}(\Omega),L_2(\Omega)]_\theta,[W_{2,D}^{1}(\Omega),W_{2,D}^{2}(\Omega)]_\theta\big)}\,\le\, 2\,c_2(\kappa,\ve) \ ,
$$
and it then remains to characterize the interpolation spaces for $\theta\in [0,1]\setminus\{1/2\}$. For this we first invoke \cite[Thm.~1.1.11]{Triebel} to obtain that
$$
[W_{2,D}^{-1}(\Omega),L_2(\Omega)]_\theta = [(W_{2,D}^{1}(\Omega))',(L_2(\Omega))']_\theta = [W_{2,D}^{1}(\Omega),L_2(\Omega)]_\theta'
$$
with equivalent norms and hence, using \cite[Cor.~1.4.4.5]{Grisvard} and \cite[Prop.~2.11 \& Prop.~3.3]{JerisonKenig} to characterize the last interpolation space,
$$
[W_{2,D}^{-1}(\Omega),L_2(\Omega)]_\theta = (W_{2,D}^{1-\theta}(\Omega))' = W_{2,D}^{\theta-1}(\Omega)\ .
$$
Finally, since the embedding
$$
[W_{2,D}^{1}(\Omega),W_{2,D}^{2}(\Omega)]_\theta \hookrightarrow W_{2,D}^{\theta+1}(\Omega)
$$
is obviously continuous by \cite[Thm.~4.3.2/2]{Triebel} and definition of interpolation, the assertion follows.
\end{proof}

Next, we show that $\mathcal{A}(v)$ is  Lipschitz continuous with respect to $v$ in a suitable topology. More precisely:

\begin{lem}\label{L1c}
Given $\xi\in [0,(q-1)/q)$ and $\alpha\in (\xi,1)$, there exists $c_6(\kappa,\ve)>0$ such that
$$
\|\mathcal{A}(v)-\mathcal{A}(w)\|_{\ml(W_{2,D}^{2}(\Omega),W_{2,D}^{-\alpha}(\Omega))}\le c_6(\kappa,\ve)\, \|v-w\|_{W_{q}^{2-\xi}(I)}\ ,\quad v,w\in \overline{S}_q(\kappa)\ .
$$
\end{lem}

\begin{proof}[{\bf Proof}]
Consider $v,w\in \overline{S}_q(\kappa)$ and $\Phi\in W_{2,D}^2(\Omega)$. Then $\mathcal{A}(v)\Phi\in L_2(\Omega)\hookrightarrow W_{2,D}^{-\alpha}(\Omega)$ and so, for any $\varphi\in W_{2,D}^\alpha(\Omega)$,
\begin{equation*}
\begin{split}
\int_\Omega \left[ \big(\mathcal{A}(v)-\mathcal{A}(w)\big)\,\Phi \right]\,\varphi\,\rd (x,\eta) 
 =\ &  2\ve^2\int_\Omega \eta\, \left(\frac{\partial_x v}{1+v}-\frac{\partial_x w}{1+w}\right)\, \partial_x\partial_\eta\Phi\,\varphi\, \rd (x,\eta)\\
&-\int_\Omega \left(\frac{1+\ve^2\eta^2(\partial_x v)^2}{(1+v)^2}-\frac{1+\ve^2\eta^2(\partial_x w)^2}{(1+w)^2}\right)\,\partial_\eta^2\Phi\,\varphi\, \rd (x,\eta)\\
&- 2\ve^2\int_\Omega\eta\left[\left(\frac{\partial_x v}{1+v}\right)^2-\left(\frac{\partial_x w}{1+w}\right)^2\right]\, \partial_\eta\Phi\,\varphi\, \rd (x,\eta)\\
&+\ve^2\int_\Omega\eta  \,\left(\frac{\partial_x^2 v}{1+v}-\frac{\partial_x^2 w}{1+w}\right)\, \partial_\eta\Phi\,\varphi\, \rd (x,\eta)\\
=\ & :J_1+J_2+J_3+J_4\ .
\end{split}
\end{equation*}
Since $\xi\in [0,(q-1)/q)$ it follows from the continuous embedding $W_q^{2-\xi}(I)\hookrightarrow W_\infty^1(I)$ that
\begin{equation*}
\begin{split}
\vert J_1\vert &\le 2\,\ve^2\,\left\|\frac{\partial_x v}{1+v}-\frac{\partial_x w}{1+w}\right\|_{L_\infty(I)}\, \|\partial_x\partial_\eta\Phi\|_{L_2(\Omega)}\, \|\varphi\|_{L_2(\Omega)}\\
&\le \ve^2\, c(\kappa)\, \left\|v-w\right\|_{W_q^{2-\xi}(I)}\, \| \Phi\|_{W_{2,D}^2(\Omega)}\, \|\varphi\|_{W_{2,D}^\alpha(\Omega)}\ .
\end{split}
\end{equation*}
Similarly, we have
$$
\vert J_2\vert + \vert J_3\vert \le (1+\ve^2)\, c(\kappa)\, \left\|v-w\right\|_{W_q^{2-\xi}(I)}\, \| \Phi\|_{W_{2,D}^2(\Omega)}\, \|\varphi\|_{W_{2,D}^\alpha(\Omega)}\ .
$$ 
Finally, we infer from the embedding  $W_2^{1}(\Omega)\hookrightarrow L_{2q/(q-2)}(\Omega)$ and the fact that $(W_{q'}^\xi(I))'=W_q^{-\xi}(I)$ since $\xi\in [0,(q-1)/q)$, see, e.g., \cite[Eq.~(5.14)]{AmannTeubner}, that
\begin{equation*}
\begin{split}
\vert J_4\vert &\le \ve^2\left\vert\int_\Omega \partial_x^2 (v-w)\,\frac{\eta\, \partial_\eta\Phi\,\varphi} {1+v}\, \rd (x,\eta)\right\vert\\
&\quad+\ve^2\,\|\partial_x^2 w\|_{L_q(I)}\left\| \frac{1}{1+v}-\frac{1}{1+w}\right\|_{L_\infty(I)}\,\|\partial_\eta\Phi\|_{L_{2q/(q-2)}(\Omega)}\, \|\varphi\|_{L_2(\Omega)}\\
&\le \ve^2\, \|\partial_x^2 (v-w)\|_{W_q^{-\xi}(I)}\, \left\|\frac{1}{1+v} \int_0^1 \eta\, \partial_\eta\Phi(\cdot,\eta)\,\varphi(\cdot,\eta) \, \rd \eta\right\|_{W_{q'}^{\xi}(I)} \\
&\quad+\ve^2\,c(\kappa)\, \| v-w\|_{W_q^{2-\xi}(I)}\, \|\partial_\eta\Phi\|_{W_{2,D}^1(\Omega)}\, \|\varphi\|_{W_{2,D}^\alpha(\Omega)}\ .
\end{split}
\end{equation*}
Now, choosing $\beta\in [\xi,\alpha)$, pointwise multiplication is continuous as mappings
$$
W_q^2(I)\cdot W_2^\beta (I)\hookrightarrow W_{q'}^{\xi}(I)\ ,\qquad C^2(\bar{\Omega})\cdot W_2^1(\Omega)\cdot W_2^\alpha(\Omega)\hookrightarrow W_2^\beta(\Omega)\ 
$$
according to Theorem~\ref{PM}, while an interpolation argument shows
\bqn\label{interpol}
\left\|\int_0^1\zeta(\cdot,\eta)\,\rd \eta\right\|_{W_2^{\gamma}(I)}\le c\,\|\zeta\|_{W_2^{\gamma}(\Omega)}\ ,\quad \zeta\in W_2^{\gamma}(\Omega)\ ,
\eqn
for all $\gamma\in [0,1]$, since $\zeta \mapsto \int_0^1\zeta(\cdot,\eta)\,\rd \eta$ belongs to both $\ml(L_2(\Omega),L_2(I))$ and $\ml(W_2^1(\Omega),W_2^1(I))$. Therefore, we deduce that
\begin{equation*}
\begin{split}
\left\|\frac{1}{1+v} \int_0^1 \eta\, \partial_\eta\Phi(\cdot,\eta)\,\varphi(\cdot,\eta) \, \rd \eta\right\|_{W_{q'}^{\xi}(I)}\,
&\le \left\|\frac{1}{1+v}\right\|_{W_q^2(I)}\, \left\|\int_0^1 \eta\, \partial_\eta\Phi(\cdot,\eta)\,\varphi(\cdot,\eta) \, \rd \eta\right\|_{W_2^{\beta}(I)}\\
&\le\, c(\kappa)\, \|\eta\,\partial_\eta\Phi\,\varphi\|_{W_{2,D}^\beta(\Omega)}\\
&\le \, c(\kappa)\, \|\partial_\eta\Phi\|_{W_{2,D}^1(\Omega)}\, \|\varphi\|_{W_{2,D}^\alpha(\Omega)}\ ,
\end{split}
\end{equation*}
and we arrive at
$$
\vert J_4\vert \le \ve^2\, c(\kappa)\, \left\|v-w\right\|_{W_q^{2-\xi}(I)}\, \| \Phi\|_{W_{2,D}^2(\Omega)}\, \|\varphi\|_{W_{2,D}^\alpha(\Omega)}\ .
$$ 
Consequently, gathering the above estimates on $\vert J_i\vert$, we obtain
$$
\|(\mathcal{A}(v)-\mathcal{A}(w))\Phi\|_{W_{2,D}^{-\alpha}(\Omega)}\le (1+\ve^2)\, c(\kappa)\, \|v-w\|_{W_{q}^{2-\xi}(I)}\,\|\Phi\|_{W_{2,D}^2(\Omega)}\ ,
$$
whence the claim.
\end{proof}

Next, we prove that $f_v$ from \eqref{f} depends Lipschitz continuously on $v\in \overline{S}_q(\kappa)$.

\begin{lem}\label{L1d}
 Given $\xi\in [0,(q-1)/q)$ and $\alpha\in (\xi,1)$,  there exists $c_7(\kappa)>0$ such that
$$
\|f_v-f_w\|_{W_{2,D}^{-\alpha}(\Omega)}\le c_7(\kappa)\, \ve^2\, \|v-w\|_{W_{q}^{2-\xi}(I)}\ ,\quad v,w\in \overline{S}_q(\kappa)\ .
$$
\end{lem}

\begin{proof}[{\bf Proof}]
Consider $v,w\in \overline{S}_q(\kappa)$ and recall that  $f_v\in L_2(\Omega)\hookrightarrow W_{2,D}^{-\alpha}(\Omega)$ by \eqref{spip}. Then, it follows from the embedding $W_2^{2-\xi}(I)\hookrightarrow W_\infty^1(I)$ and \eqref{2.6} that, for $\varphi\in W_{2,D}^\alpha(\Omega)$,
\begin{equation*}
\begin{split}
&\left\vert \int_\Omega \big(f_v-f_w\big)\,\varphi\,\rd (x,\eta)\right\vert\\
&\qquad\le\, 2\,\ve^2\int_\Omega \left\vert\left(\frac{\partial_xv}{1+v}\right)^2-\left(\frac{\partial_xw}{1+w}\right)^2\right\vert\, \vert\varphi\vert\,\rd (x,\eta)
 + \ve^2\int_\Omega \left\vert\partial_x^2 v\right\vert\, \left\vert\frac{1}{1+v}-\frac{1}{1+w}\right\vert\,\vert\varphi\vert\,\rd (x,\eta)\\
&\qquad\quad + \ve^2\left\vert\int_\Omega \big(\partial_x^2 v-\partial_x^2 w\big)\,\frac{\varphi}{1+w}\,\rd (x,\eta)\right\vert\\
&\qquad\le 
\ve^2\, c(\kappa)\, \|v-w\|_{W_\infty^1(I)}\, \|\varphi\|_{L_2(\Omega)} + \ve^2\, c(\kappa)\, \|\partial_x^2v\|_{L_2(I)}\, \|v-w\|_{L_\infty(I)}\, \|\varphi\|_{L_2(\Omega)}\\
&\qquad\quad + \ve^2\, \|\partial_x^2(v-w)\|_{W_q^{-\xi}(I)}\,\left\|\frac{1}{1+w}\int_0^1\varphi(\cdot,\eta)\,\rd \eta \right\|_{W_{q'}^\xi(I)}\\
&\qquad\le\, \ve^2\,c(\kappa)\, \|v-w\|_{W_q^{2-\xi}(I)}\left(\|\varphi\|_{W_{2,D}^\alpha(\Omega)}
+\left\|\frac{1}{1+w}\right\|_{W_q^2(I)}\, \left\|\int_0^1\varphi(\cdot,\eta)\,\rd \eta \right\|_{W_2^\alpha(I)}\right)\ ,
\end{split}
\end{equation*}
where we have used for the last inequality continuity of pointwise multiplication
$$
W_q^2(I)\cdot W_2^\alpha(I)\hookrightarrow W_{q'}^\xi(I)
$$
as in the proof of Lemma~\ref{L1c}. Taking $\gamma=\alpha$ in \eqref{interpol} and using \eqref{2.6}, we end up with
$$
\left\vert \int_\Omega \big(f_v-f_w\big)\,\varphi\,\rd (x,\eta)\right\vert\,\le\,\ve^2\,c(\kappa)\,\|v-w\|_{W_q^{2-\xi}(I)} \, \|\varphi\|_{W_{2,D}^\alpha(\Omega)}\ ,
$$
which yields the assertion.
\end{proof}

Combining the previous lemmas we now readily obtain the Lipschitz continuity of $\phi_v$ with respect to $v\in \overline{S}_q(\kappa)$.

\begin{lem}\label{L2a}
Given $\xi\in [0,(q-1)/q)$ and $\alpha\in (\xi,1)$,  there exists $c_8(\kappa,\ve)>0$ such that
$$
\|\phi_v-\phi_w\|_{W_{2,D}^{2-\alpha}(\Omega)}\le c_8(\kappa,\ve)\, \|v-w\|_{W_{q}^{2-\xi}(I)}\ ,\quad v,w\in \overline{S}_q(\kappa)\ .
$$
\end{lem}

\begin{proof}[{\bf Proof}]
Let $v,w\in \overline{S}_q(\kappa)$ and recall that $\phi_v=\Phi_v +\eta$ with
$\Phi_v=\mathcal{A}(v)^{-1} f_v\in W_{2,D}^{2}(\Omega)$. Then obviously $\phi_v-\phi_w {  = \Phi_v - \Phi_w} \in W_{2,D}^{2-\alpha}(\Omega)$ and so
\begin{equation*}
\begin{split}
\|\phi_v-\phi_w\|_{W_{2,D}^{2-\alpha}(\Omega)}
&=\| \mathcal{A}(v)^{-1} (f_v-f_w) +(\mathcal{A}(v)^{-1} -\mathcal{A}(w)^{-1}) f_w\|_{W_{2,D}^{2-\alpha}(\Omega)} \\
&\le\, 
\| \mathcal{A}(v)^{-1}\|_{\ml(W_{2,D}^{-\alpha}(\Omega),W_{2,D}^{2-\alpha}(\Omega))}\, \| f_v-f_w\|_{W_{2,D}^{-\alpha}(\Omega)}\\
&\quad   + \|\mathcal{A}(v)^{-1}\|_{\ml(W_{2,D}^{-\alpha}(\Omega),W_{2,D}^{2-\alpha}(\Omega))}\, \|\mathcal{A}(v)  -\mathcal{A}(w)\|_{\ml(W_{2,D}^{2}(\Omega),W_{2,D}^{-\alpha}(\Omega))}\\
&\qquad\qquad \qquad\qquad\times \|\mathcal{A}(w)^{-1}\|_{\ml(L_2(\Omega),W_{2,D}^{2}(\Omega))}\, \|f_w\|_{L_2(\Omega)}
\end{split}
\end{equation*}
and the statement follows from \eqref{spip}, Lemma~\ref{L1b}, Lemma~\ref{L1c}, and Lemma~\ref{L1d}.
\end{proof}

\begin{proof}[{\bf Proof of Proposition~\ref{L1}}]
As previously noted, we are left to prove assertion \eqref{gLip}. For this let  {  $q\in (2,\infty)$, $\xi\in [0,1/2)$, and $\nu\in [0,(1-2\xi)/2)$. By \cite[Eq.~(2.30)]{ELW_CPAM}, given $2\sigma\in [0,1/2)$ there is $c_9(\kappa,\ve)$ such that
\bqn\label{q1}
\|\partial_\eta \phi_v(.,1) \|_{W_2^{1/2}(I)} + \| \vert\partial_\eta\phi_v(\cdot,1)\vert^2\|_{W_2^{2\sigma}(I)}\le c_9(\kappa,\ve)\ ,\quad v\in \overline{S}_q(\kappa)\ .
\eqn
Given $v, w \in \overline{S}_q(\kappa)$, it follows from the definition of $g_\ve$ that
\begin{equation*}
\begin{split}
  \|g_\ve(v)-& g_\ve(w)\|_{W_2^\nu(I)} \nonumber\\
&\le  \ve^2 \left\| \frac{\partial_x v + \partial_x w}{(1+v)^2}\, (\partial_xv-\partial_xw)\,  \vert\partial_\eta\phi_v(\cdot,1)\vert^2 \right\|_{W_2^{\nu}(I)} \nonumber\\
& \quad + \left\| \left( 1+\ve^2(\partial_xw)^2 \right)\, \frac{2+v+w}{(1+v)^2 (1+w)^2}\, (v-w)\, \vert\partial_\eta\phi_v(\cdot,1)\vert^2 \right\|_{W_2^{\nu}(I)} \nonumber\\
& \quad + \left\| \frac{1+\ve^2(\partial_xw)^2}{(1+w)^2}\, \left( \partial_\eta\phi_v(\cdot,1)+\partial_\eta\phi_w(\cdot,1) \right)\, \left(  \partial_\eta\phi_v(\cdot,1)- \partial_\eta\phi_w(\cdot,1) \right) \right\|_{W_2^{\nu}(I)} \nonumber\\
& =:  J_1 + J_2 + J_3 \ . \label{q1bis}
\end{split}
\end{equation*}
We now fix $2\sigma\in (\xi+\nu,1/2)$ and $s\in [\nu,1-\xi)$ with $s\ge 1/q$ so that pointwise multiplication
\begin{equation}
W_q^s(I) \cdot W_2^{2\sigma}(I) \hookrightarrow W_2^\nu(I) \label{q1ter}
\end{equation}
is continuous by Theorem~\ref{PM}. Then, 
$$
J_1 \le c\,\ve^2\, \left\| \frac{\partial_x v + \partial_x w}{(1+v)^2}\, (\partial_xv-\partial_xw) \right\|_{W_q^s(I)}\, \left\| \vert\partial_\eta\phi_v(\cdot,1)\vert^2 \right\|_{W_2^{2\sigma}(I)}\ .
$$
We infer from \eqref{q1} and continuity of pointwise multiplication
\bqn\label{q2}
W_q^2(I)\cdot W_q^1(I)\cdot W_q^{1-\xi}(I) \hookrightarrow W_q^s(I)
\eqn
guaranteed by Theorem~\ref{PM} that
$$
J_1 \le c(\kappa,\ve)\, \left\|\frac{1}{(1+v)^2}\right\|_{W_q^2(I)}\, \|\partial_xv+\partial_xw\|_{W_q^1(I)}\, \|\partial_xv-\partial_xw\|_{W_q^{1-\xi}(I)}\ ,
$$
hence, since both $v$ and $w$ belong to $\overline{S}_q(\kappa)$ and $W_q^2(I)$ is an algebra, 
\begin{equation}
J_1 \le c(\kappa,\ve)\, \| v - w \|_{W_q^{2-\xi}(I)}\ . \label{q1quattro}
\end{equation}
Using again the continuity \eqref{q1ter} and \eqref{q2} of pointwise multiplication as well as \eqref{q1}, we obtain
$$
J_2 \le c(\kappa,\ve)\, \left\| 1+\ve^2(\partial_xw)^2 \right\|_{W_q^1(I)}\, \left\| \frac{2+v+w}{(1+v)^2 (1+w)^2}\right\|_{W_q^2(I)}\, \|v-w\|_{W_q^{1-\xi}(I)} \ .
$$
As $W_q^1(I)$ and $W_q^2(I)$ are algebras, we deduce from the properties of $\overline{S}_q(\kappa)$ that
\begin{equation}
J_2 \le c(\kappa,\ve)\, \| v - w \|_{W_q^{2-\xi}(I)}\ . \label{q1cinque}
\end{equation}
Finally, fix $\alpha\in(\xi, 2\sigma-\nu)$. Invoking once more Theorem~\ref{PM} gives continuity of pointwise multiplication
\bqn\label{q3} 
W_q^1(I)\cdot W_2^{1/2}(I)\cdot W_2^{1/2-\alpha}(I) \hookrightarrow W_2^\nu(I)\ ,
\eqn
and thus
$$
J_3 \le \left\|\frac{1+\ve^2(\partial_xw)^2}{(1+w)^2}\right\|_{W_q^1(I)}\,   \| \partial_\eta\phi_v(\cdot,1)+\partial_\eta\phi_w(\cdot,1)\|_{W_2^{1/2}(I)}\, \| \partial_\eta\phi_v(\cdot,1)-\partial_\eta\phi_w(\cdot,1)\|_{W_2^{1/2-\alpha}(I)}\ .
$$
Since $w\in \overline{S}_q(\kappa)$, it follows from \eqref{q1}, the properties of the trace operator \cite[Thm.~1.5.1.1]{Grisvard}, and Lemma~\ref{L2a} that
\begin{equation}
J_3 \le c(\kappa,\ve)\, \left\| \phi_v - \phi_w \right\|_{W_{2,D}^{2-\alpha}(\Omega)} \le c(\kappa,\ve)\, \|v-w\|_{W_{q}^{2-\xi}(I)}\ . \label{q1sei}
\end{equation}
Gathering \eqref{q1quattro}, \eqref{q1cinque}, and \eqref{q1sei} completes the proof of Proposition~\ref{L1}.}
\end{proof}

\section{Well-Posedness of the Evolution Problem: Proof of Theorem~\ref{A} }\label{Sec3} 

\medskip

We next focus our attention on the quasilinear curvature part of equation \eqref{u}. Let us first note that
\bqn\label{curv}
\partial_x\left(\frac{\partial_x u}{\sqrt{1+\varepsilon^2(\partial_x u)^2}}\right)= \frac{1}{\big(1+\varepsilon^2(\partial_x u)^2\big)^{3/2}}\, \partial_x^2 u\ ,
\eqn
which motivates the definition of
\begin{equation}
A(w)v:=-\frac{1}{\big(1+(\partial_x w)^2\big)^{3/2}}\,\partial_x^2 v\ ,\quad v\in W_{q,D}^2(I)\ , \label{opA}
\end{equation}
where $w$ belongs, for $q\in (2,\infty)$, $\kappa\in (0,1)$, and $\xi\in (0,(q-1)/q)$ given, to the set
$$
Z_q(\kappa):=\{w\in W_q^{2-\xi}(I)\,;\, \|w\|_{W_q^{2-\xi}(I)}\le 1/\kappa\}\ .
$$
Note that $W_q^{2-\xi}(I)\hookrightarrow C^1([-1,1])$. Also observe that $-A(\ve u)u$ coincides with \eqref{curv}. For $\omega>0$ and $k\ge 1$,  let
$\mathcal{H}(W_{q,D}^2(I),L_q(I);k,\omega)$ be the set of all $\mathcal{A}\in\mathcal{L}(W_{q,D}^2(I),L_q(I))$ such that $\omega+\mathcal{A}$ is an isomorphism from $W_{q,D}^2(I)$ onto $L_q(I)$ and satisfies the resolvent estimates
$$
\frac{1}{k}\,\le\,\frac{\|(\mu+\mathcal{A})z\|_{L_q(I)}}{\vert\mu\vert \,\| z\|_{L_q(I)}+\|z\|_{W_{q,D}^2(I)}}\,\le \, k\ ,\quad Re\, \mu\ge \omega\ ,\quad z\in W_{q,D}^2(I)\setminus\{0\}\ .
$$
Then $\mathcal{A}\in\mathcal{H}(W_{q,D}^2(I),L_q(I);k,\omega)$ implies {  that $\mathcal{A}\in\mathcal{H}(W_{q,D}^2(I),L_q(I))$, that is,} $-\mathcal{A}$  generates an analytic semigroup on $L_q(I)$ with domain $W_{q,D}^2(I)$, see \cite[I.Thm.1.2.2]{LQPP}. 

\begin{lem}\label{SG}
Let $q\in (2,\infty)$, $\kappa\in (0,1)$, and $\xi\in (0,(q-1)/q)$. Then there are $k:=k(\kappa)\ge 1$ {and $\omega:=\omega(\kappa) >0$} such that, for each $w\in Z_q(\kappa)$,
$$
{-2\omega+A(w)\in \mathcal{H}(W_{q,D}^2(I),L_q(I);k,\omega)}
$$
and $A(w)$ is resolvent positive. Moreover, there is a constant $\ell(\kappa)>0$ such that
\bqn\label{esti}
\| A(w_1)-A(w_2)\|_{\ml(W_{q,D}^2(I),L_q(I))}\le \ell(\kappa) \, \|w_1-w_2\|_{W_q^{2-\xi}(I)}\ ,\quad w_1, w_2\in Z_q(\kappa)\ .
\eqn
\end{lem}

\begin{proof}[{\bf Proof}]
Note that the continuous embedding $W_q^{2-\xi}(I)\hookrightarrow C^1([-1,1])$ ensures the existence of $c(\kappa)>1$ such that
\bqn\label{ck}
1\le (1+(\partial_x w)^2)^{3/2}< c(\kappa)\ ,\quad w\in Z_q(\kappa)\ .
\eqn
Let $w\in Z_q(\kappa)$ be fixed and put $W:= (1+(\partial_x w)^2)^{3/2}$.
Since e.g. \cite[Rem.4.2(c)]{AmannTeubner} ensures that $-A(w)$ generates an analytic semigroup on $L_q(I)$, the operator $\mu+A(w)$ is boundedly invertible for each $\mu\in\C$ with sufficiently large real part. Clearly, $\mu$ belongs to the spectrum of $-A(w)$ only if it is an eigenvalue, and if $\varphi\in W_{q,D}^2(I,\C)$ is a corresponding eigenvector, then testing the equation 
$$
\mu \varphi-\frac{1}{W}\partial_x^2\varphi=0
$$ 
with $ W\bar{\varphi}$ readily gives $\mu \in (-\infty,0 )$ and $\mu\le -\mu_1/\| W\|_\infty$, where $\mu_1:=\pi^2/4$ denotes the principal eigenvalue of $-\partial_x^2$ subject to Dirichlet boundary conditions. Consequently, due to \eqref{ck} the resolvent set of $-A(w)$ contains the half plane $[\mathrm{Re}\,\mu \ge -\mu_1/c(\kappa)]$. Put $\omega:=\mu_1/(2^{1+q/2} c(\kappa))>0$. Then $\omega$ is independent of $w$ and $-\omega+A(w)$ is an isomorphism from $W_{q,D}^2(I)$ onto $L_q(I)$. Moreover, if $F\in L_q(I,\C)$ and $\mathrm{Re}\,\mu >0$, then the equation
$$
\mu u-  2\omega  u-\frac{1}{W}\partial_x^2u=F
$$ 
has a unique solution $u\in W_{q,D}^2(I,\C)$. Testing this equation by $W\vert u\vert^{q-2}\bar{u}$ and using \eqref{ck}  yield the resolvent estimate  (see, e.g. the proof of \cite[Prop.2.4.2]{LunardiInternetSeminar} for details)
$$
\|(\mu-{  2\omega}+A(w))^{-1}\|_{\ml(L_q(I))}{  
\le} \frac{c'(\kappa)}{\vert\mu\vert} \ ,
$$
where $c'(\kappa)>0$ is a constant independent of $w$.
Since we clearly have by \eqref{ck}
$$
\|-  2\omega+A(w)\|_{\ml(W_{q,D}^2(I),L_q(I))}\le   c'(\kappa)\ ,
$$
it follows from \cite[I.Rem.1.2.1(a)]{LQPP} that 
$$
- 2\omega+A(w)\in \mathcal{H}(W_{q,D}^2(I),L_q(I);k, \omega)
$$
for some $k:=k(\kappa)\ge 1$. That $A(w)$ is resolvent positive and Lipschitz property \eqref{esti} are readily seen, the latter being a consequence of the continuous embedding $W_q^{2-\xi}(I)\hookrightarrow C^1([-1,1])$.
\end{proof}

If $w=w(t)$ is H\"older continuous in $t$, then $A(w)$ generates a parabolic evolution operator $U_{A(w)}(t,s)$ on $L_q(I)$ in the sense of \cite[Sect.~II]{LQPP}. More precisely:

\begin{prop}\label{ES}
Let $q\in (2,\infty)$, $\kappa\in (0,1)$, and $\xi\in (0,(q-1)/q)$. Let {$\omega(\kappa)>0$ and} $\ell(\kappa)>0$ be as in Lemma~\ref{SG} and define, for $\rho\in (0,1)$ and $N,\tau>0$ given,
\begin{equation*}
\begin{split}
\mathcal{W}_\tau(\kappa):=\{w\in C([0,\tau],W_{q,D}^{2-\xi}(I))\,;\,\|&w(t)-w(s)\|_{W_{q,D}^{2-\xi}(I)}\le\frac{N}{\ell(\kappa)} \vert t-s\vert^\rho\\
&\text{and  } w(t)\in Z_q(\kappa)\ \text{for } 0\le t,s\le \tau\}\ .
\end{split}
\end{equation*}
Then, there is a constant $c_0(\rho)>0$ being independent of $N$ and $\tau$ such that the following is true: for each  $w\in \mathcal{W}_\tau(\kappa)$, there exists a unique parabolic evolution operator $U_{A(w)}(t,s)$, $0\le s\le t\le \tau$, possessing $W_{q,D}^{2}(I)$ as a regularity subspace, and
$$
\|U_{A(w)}(t,s)\|_{\mathcal{L}(W_{q,D}^{2\alpha}(I),W_{q,D}^{2\beta}(I))} \le c_*(\kappa)\, (t-s)^{\alpha-\beta}\, e^{-\vartheta (t-s)}\ ,\quad 0\le s < t\le \tau\ ,
$$
for $0\le \alpha\le\beta\le 1$ with $2\alpha, 2\beta\not= 1/q$, where the constant $c_*(\kappa)\ge 1$ depends on $N$, $\alpha$, and $\beta$ but is independent of~$\tau$, and
\bqn\label{vartheta}
-\vartheta:=-{  \vartheta(\kappa, \rho, N)}:=c_0(\rho) N^{1/\rho}-{ \omega(\kappa)}\ .
\eqn 
Moreover, $U_{A(w)}(t,s)\in \ml(L_q(I))$ is  a positive operator for $0\le s\le t\le \tau$.
\end{prop}

\begin{proof}[{\bf Proof}]
Notice that, for each  $w\in \mathcal{W}_\tau(\kappa)$,
\bqn\label{lqpp1}
A(w)\in C^\rho([0,\tau],\ml(W_{q,D}^{2}(I),L_q(I)))\ ,\quad { -2\omega(\kappa)+A(w)\subset \mathcal{H}(W_{q,D}^2(I),L_q(I);k(\kappa),\omega(\kappa))}
\eqn
with 
\bqn\label{lqpp2}
\sup_{0\le s<t\le \tau}\frac{\| A(w(t))-A(w(s))\|_{\ml(W_{q,D}^2(I),L_q(I))}}{\vert t-s\vert^\rho}\le N
\eqn
 by Lemma~\ref{SG}. Hence, the assertion follows from \cite[II.~Thm.~5.1.1, Lem.~5.1.3, Thm.~6.4.2]{LQPP} and the interpolation results of~\cite{Grisvard69, Seeley}.
\end{proof}

We are now in a position to prove Theorem~\ref{A}.

\begin{proof}[{\bf Proof of Theorem~\ref{A}}]

Using the previously introduced notation, we may rewrite \eqref{33} subject to the boundary condition \eqref{bcu} and the initial condition \eqref{ic} as an abstract $\lambda$-dependent quasilinear Cauchy problem of the form
\bqn\label{CP}
\frac{\rd}{\rd t}u+A(\ve u)u=-\lambda g_\ve(u)\ ,\quad t>0\ ,\qquad u(0)=u^0\ .
\eqn
Existence of solutions to \eqref{CP} is based on applying a fixed point argument.

Let $\lambda> 0$ and $q\in (2,\infty)$. For simplicity we restrict to the case  $\varepsilon\in (0,1]$. A similar proof works for $\ve>1$ by changing some of the constants occurring in the proof. Let us consider an initial value $u^0\in  W_{q,D}^2(I)$  with $u^0(x)>-1$ for $x\in I$. Clearly, there is $\kappa\in (0,1/2)$ with 
\bqn\label{44}
u^0\in  \overline{S}_q(2\kappa)\ ,\quad \|u^0\|_{W_{q,D}^{2-\xi}(I)}\le \frac{1}{2\kappa}\ ,
\eqn
where we fix $\xi$ and $\sigma$ such that
$$
0<\xi<\frac{1}{q}\ ,\quad 0< \frac{1}{2}-\frac{1}{q}<2\sigma<\frac{1}{2}-\xi\ .
$$ 
Let $4\rho\in (0,\xi)$, let $c_0(\rho)>0$ be as in Proposition~\ref{ES} and choose then $N>0$ with the property that $-\vartheta:=c_0(\rho) N^{1/\rho}- \omega(\kappa)<0$ in \eqref{vartheta}. Since  
\bqn\label{embb}
W_{2,D}^{2\sigma}(I)\hookrightarrow W_{q,D}^{2\sigma-\frac{1}{2}+\frac{1}{q}}(I)\hookrightarrow L_q(I)\ ,
\eqn
it follows from Proposition~\ref{ES} (with $(\alpha,\beta)=(1,1)$ and ($\alpha,\beta)=(\sigma-(q-2)/4q,1)$) that, for $w\in\mathcal{W}_\tau(\kappa)$ fixed,
\bqn\label{UA}
\|U_{A( w)}(t,s)\|_{\mathcal{L}(W_{q,D}^{2}(I))} + {  (t-s)^{-\sigma+1+\frac{1}{2}(\frac{1}{2}-\frac{1}{q})}}\, \|U_{A( w)}(t,s)\|_{\mathcal{L}(W_{2,D}^{2\sigma}(I),W_{q,D}^{2}(I))} \le c_*(\kappa)\, \, e^{-\vartheta (t-s)}\ ,
\eqn
for $0\le s \le t\le \tau$, where the constant $c_*(\kappa)\ge 1$ is independent of $w$ and $\tau>0$.
Now, set
\begin{equation*}
\mathcal{V}_\tau(\kappa) := \Big\{v\in \mathcal{W}_\tau(\kappa)\,;\,  \|v(t)\|_{W_{q,D}^2(I)}\le \frac{1}{\kappa_0}\ \text{and }v(t)\ge -1+\kappa\ \text{for } 0\le t,s\le \tau\Big\}\ ,
\end{equation*}
with $\kappa_0:= \kappa/c_*(\kappa)\le \kappa$ and $\tau>0$ to be chosen later and observe that, when endowed with the topology of $C([0,\tau],W_{q,D}^{2-\xi}(I))$, $\mathcal{V}_\tau(\kappa)$ is a complete metric space. In addition, since $\kappa_0\le \kappa$, we have $v(t)\in \overline{S}_q(\kappa)$ for all $t\in 0,\tau]$ and $v\in \mathcal{V}_\tau(\kappa)$. It is, moreover, worthwhile to point out that $\ve v\in \mathcal{V}_\tau(\kappa)$ for $v\in \mathcal{V}_\tau(\kappa)$ since $\ve\in (0,1]$.\footnote{As already mentioned, the case $\ve>1$ can be handled by taking different values of $\kappa$ and $N$ in the definition of $\mathcal{V}_\tau(\kappa)$.}
By Proposition~\ref{L1} there is $c_{1}(\kappa,\varepsilon)>0$ such that
\bqn\label{10}
\|g_\ve(v)-g_\ve(w)\|_{W_{2,D}^{2\sigma}(I)}\le c_{1}(\kappa,\varepsilon)\ \|v-w\|_{W_{q,D}^{2-\xi}(I)}\ ,\quad v, w\in  \overline{S}_q\left( \kappa \right)\ ,
\eqn
and
\bqn\label{10a}
\|g_\ve(v)\|_{W_{2,D}^{2\sigma}(I)}\le  c_{1}(\kappa,\varepsilon)\ ,\quad v\in { \overline{S}_q}\left( \kappa \right)\ .
\eqn
We then claim that
$$
\Lambda(v)(t):=U_{A(\ve v)}(t,0)\,u^0-\lambda\int_0^t U_{A(\ve v)}(t,s)\, g_\ve\big(v(s)\big)\,\rd s\ ,\quad t\in [0,\tau]\ ,\quad v\in \mathcal{V}_\tau(\kappa)\ ,
$$
defines a contraction from $\mathcal{V}_\tau(\kappa)$ into itself if either $\lambda>0$ is arbitrary and $\tau= \tau(\kappa,\lambda)>0$ is sufficiently small, or $\lambda>0$ is sufficiently small and $\tau>0$ is arbitrary. To see this let $v, w$ be arbitrary elements of $\mathcal{V}_\tau(\kappa)$ and let $t\in [0,\tau]$. Since {$U_{A(\ve v)}(t,0)$} is a positive operator and $u^0\ge -1+2\kappa$, it follows from the embedding $W_{q}^{2}(I)\hookrightarrow L_\infty(I)$ with constant $2$, \eqref{UA}, and \eqref{10a} that
\begin{align}
\Lambda(v)(t)\ge & -1 + 2\kappa - 2\,\lambda\,  \int_0^t \left\|  U_{A(\ve v)}(t,s)\, g_\ve(v(s)) \right\|_{W_{q,D}^2(I)}\, \rd s \nonumber \\
\ge & -1+2\kappa -  2\,\lambda\, c_*(\kappa)\,  \ \int_0^t e^{-\vartheta (t-s)}\ (t-s)^{\sigma-1-\frac{1}{2}(\frac{1}{2}-\frac{1}{q})}\ \|g_\ve(v(s))\|_{W_{2,D}^{2\sigma}(I)}\,\rd s \nonumber\\
\ge & -1+2\kappa -  2\,\lambda\,  c_*(\kappa)\, c_{1}(\kappa,\varepsilon) \int_0^\tau e^{-\vartheta  s}\ s^{\sigma-1-\frac{1}{2}(\frac{1}{2}-\frac{1}{q})}\,\rd s \ , \label{zui}
\end{align}
while \eqref{UA} and \eqref{10a} ensure that
\begin{align}
\| \Lambda(v)(t) \|_{W_{q,D}^{2}(I)} & \le  c_*(\kappa)\, \|u^0\|_{W_{q,D}^2(I)} \nonumber\\
&\quad + \lambda\, c_*(\kappa) \int_0^t e^{-\vartheta(t-s)}\ (t-s)^{\sigma-1-\frac{1}{2}(\frac{1}{2}-\frac{1}{q})}\ \|g_\ve(v(s))\|_{W_{2,D}^{2\sigma}(I)}\,\rd s \nonumber \\
&\le  \frac{c_*(\kappa)}{2\kappa}+\lambda\ c_*(\kappa)\ c_{1}(\kappa,\varepsilon) \int_0^\tau e^{-\vartheta  s}\ s^{\sigma-1-\frac{1}{2}(\frac{1}{2}-\frac{1}{q})}\,\rd s\ .\label{11}
\end{align}
Moreover, since $\xi>0$ and owing to \eqref{lqpp1}, \eqref{lqpp2}, and \cite[II.~Thm.~5.2.1]{LQPP} (with the choice $\beta=1-\xi/2$, $\alpha=1$, and $2\gamma=2\sigma-1/2+1/q$ of the parameters therein) there is a number $n_*(\kappa)>0$ such that 
\begin{equation*}
\begin{split}
&\|\Lambda(v)(t)-\Lambda(w)(t)\|_{W_{q,D}^{2-\xi}(I)}\\
 &\le n_*(\kappa)\, e^{-\vartheta t}\,\Bigg\{\lambda\, t^\frac{\xi}{2}\, \|g_\ve(v)-g_\ve(w)\|_{L_\infty((0,t), L_q(I))} \\
 & +t^\frac{\xi}{2}\, \|A(\ve v)-A(\ve w)\|_{C([0,\tau],\ml(W_{q,D}^2(I),L_q(I)))}\,\left[\|u^0\|_{W_{q,D}^2(I)}\,+\,\lambda\, t^{\sigma-\frac{1}{2}(\frac{1}{2}-\frac{1}{q})}\, \|g_\ve(v)\|_{L_\infty((0,t), W_{q,D}^{2\sigma-1/2+1/q}(I))}\right]     \Bigg\}\\
 &\le
  \lambda\, n_*(\kappa)\, t^\frac{\xi}{2}\, e^{-\vartheta t}\, c_1(\kappa,\ve)\, \|v-w\|_{\mathcal{V}_\tau(\kappa)}\\
 & \quad+n_*(\kappa)\,  t^\frac{\xi}{2}\, e^{-\vartheta t} \, \ell(\kappa)\, \|v-w\|_{\mathcal{V}_\tau(\kappa)}\,\left[\|u^0\|_{W_{q,D}^2(I)}\,+\,\lambda\, t^{\sigma-\frac{1}{2}(\frac{1}{2}-\frac{1}{q})}\, c_1(\kappa,\ve)\right]\ ,
\end{split}
\end{equation*}
where we used \eqref{esti}, \eqref{embb}, \eqref{10}, and \eqref{10a}  for the second inequality (recall that $\|\cdot\|_{\mathcal{V}_\tau(\kappa)} = \|\cdot\|_{C([0,\tau],W^{2-\xi}_q(I)}$). Thus, there is $c(\kappa)>0$ such that
\begin{equation}
\begin{split}\label{12}
\|\Lambda(v)(t)-\Lambda(w)(t)\|_{W_{q,D}^{2-\xi}(I)}
 &\le c(\kappa)\, \left\{ \left(\max_{0\le r\le\tau} r^\frac{\xi}{2} e^{-\vartheta r}\right)\, \left[\|u^0\|_{W_{q,D}^2(I)}+\lambda\right]\right.\\
 &\qquad\left. +\,\lambda\, \left(\max_{0\le r\le\tau} r^{\frac{\xi}{2}+\sigma-\frac{1}{2}(\frac{1}{2}-\frac{1}{q})}e^{-\vartheta r}\right)\, \right\}\,  \|v-w\|_{\mathcal{V}_\tau(\kappa)}
  \ .
\end{split}
\end{equation}
We next observe that \eqref{lqpp1}, \eqref{lqpp2}, and \cite[II.~Thm.~5.3.1]{LQPP} ensure the existence of  a number $m_*(\kappa)>0$ such that, for $0\le s\le t\le \tau$, (since $\xi>0$ and $u^0\in W_{q,D}^2(I)\hookrightarrow W_{q,D}^{2-\xi+4\rho}(I)$)
\bqn\label{122}
\begin{split}
\|\Lambda(v)(t)&-\Lambda(v)(s)\|_{W_{q,D}^{2-\xi}(I)}\\
&\le m_*(\kappa)\,   (t-s)^{2\rho}\, {  e^{-\vartheta t}}\, \left(\|u^0\|_{W_{q,D}^{2-\xi+4\rho}(I)} +\lambda\, \|g_\ve(v)\|_{L_\infty((0,t), W_{2,D}^{2\sigma}(I))}\right)\\
&\le m_*(\kappa)\,  \left(\max_{0\le r\le\tau} r^{\rho}\, e^{-\vartheta r}\right)\, \left(\|u^0\|_{W_{q,D}^{2-\xi+4\rho}(I)} +\lambda\, c_1(\kappa,\ve)\right)\, (t-s)^{\rho}
\end{split}
\eqn
by using \eqref{10a}. Similarly, from \eqref{44},
\bqn\label{122a}
\begin{split}
\|\Lambda(v)(t)\|_{W_{q,D}^{2-\xi}(I)}& \le\|\Lambda(v)(t)-\Lambda(v)(0)\|_{W_{q,D}^{2-\xi}(I)} +\|u^0\|_{W_{q,D}^{2-\xi}(I)}\\
&\le m_*(\kappa)\,   \left(\max_{0\le r\le\tau} r^{2\rho}\, e^{-\vartheta r}\right)\, \left(\|u^0\|_{W_{q,D}^{2-\xi+4\rho}(I)} +\lambda\, c_1(\kappa,\ve)\right)\, + \frac{1}{2\kappa}\ .
\end{split}
\eqn
Gathering \eqref{zui}-\eqref{122a} we see that, for arbitrary $\lambda>0$, we may choose $\tau:=\tau(\kappa,\lambda)>0$ sufficiently small such that the mapping $\Lambda :\mathcal{V}_\tau(\kappa)\rightarrow \mathcal{V}_\tau(\kappa)$ defines a contraction and thus has a unique fixed point $u$ in $\mathcal{V}_\tau(\kappa)$.

Now observe that, owing to Lemma~\ref{SG}, \eqref{embb}, and \eqref{10} we have
$$
\big(A(\ve u),g_\ve(u)\big)\in C^\rho\big([0,\tau],\mathcal{H}(W_{q,D}^2(I),L_q(I))\times W_{q,D}^{2\sigma-\frac{1}{2}+\frac{1}{q}}(I)\big)
$$
where $2\sigma-\frac{1}{2}+\frac{1}{q}>0$,
and $u$ is a mild solution to \eqref{CP} on $[0,\tau]$ with $u^0\in W_{q,D}^2(I)$. Thus,
$$
u\in C^1\big([0,\tau],L_q(I)\big)\cap C\big([0,\tau], W_{q,D}^2(I)\big)
$$ 
is a strong solution to \eqref{CP} by \cite[Thm.~4.2]{AmannTAMS} and \cite[Thm.~10.1]{AmannTeubner}, which then clearly can be extended to a maximal solution on some maximal interval $[0,T_m^\ve)$. This proves part~(i) of Theorem~\ref{A}. Since $\tau$ above depends only on $\kappa$ and $\lambda$, we also obtain part~(ii) of Theorem~\ref{A} while the proof of part~(iii) is the same as in \cite[Thm.~1.1~(iii)]{ELW_CPAM}. To prove part~(iv) of Theorem~\ref{A} we note that, since $\vartheta>0$, there are $\lambda_*(\kappa)>0$ and $r(\kappa)>0$ such that, according to \eqref{zui}-\eqref{122a}, the mapping $\Lambda :\mathcal{V}_\tau(\kappa)\rightarrow \mathcal{V}_\tau(\kappa)$ defines a contraction for each $\tau>0$ provided that  $\lambda\in (0,\lambda_*(\kappa))$ and $\|u^0\|_{W_{q,D}^{2}(I)}\le r(\kappa)$. Thus, in this case there is
a unique fixed point $u$ of $\Lambda$ belonging to $\mathcal{V}_\tau(\kappa)$ for each $\tau>0$. By definition of $\mathcal{V}_\tau(\kappa)$, this implies Theorem~\ref{A}~(iv).
\end{proof}

\begin{rem} \label{Remm}
Compared to the semilinear case investigated in \cite{ELW_CPAM}, where $A$ is independent of $v$, the quasilinear problem \eqref{CP} features a further and not negligible difficulty. Indeed, employing Banach's fixed point theorem in the proof of Theorem~\ref{A} requires two essential ingredients. First, to warrant the existence of a corresponding evolution operator $U_{A(\ve v)}$, the operator $A(\ve v(t))\in\ml (W_{q,D}^2(I),L_q(I))$ has to be a H\"older continuous function of time as shown in \eqref{122}. And second, the mapping $\Lambda=\Lambda(v)$ has to depend Lipschitz continuously (with Lipschitz constant less than 1) on its argument $v$, see \eqref{12}. Both of these properties can be guaranteed only in the topology of $W_{q,D}^{2-\xi}(I)$ for $v=v(t)$ with $\xi>0$ (but not for $\xi=0$), see \cite[II.~Thm.~5.2.1, Thm.~5.3.1]{LQPP}. This issue is the reason for refining \cite[Proposition~2.1]{ELW_CPAM} and deriving \eqref{gLip}. Moreover, unlike the semilinear case the initial value $u^0$ plays an additional role in the quasilinear case when it comes to global existence as becomes apparent from \eqref{12} and \eqref{122}. Finally, it is worthwhile to point out that solving the elliptic problem \eqref{230}, \eqref{240} requires that $v\in W_{q,D}^2(I)$, pointwise with respect to time, rather than $v\in W_{q,D}^{2-\xi}(I)$ .
\end{rem}
\section{Asymptotically Stable Steady-State Solutions: Proof of Theorem~\ref{TStable}}\label{SectTStable}

We now prove Theorem~\ref{TStable}(i).  For this let $q\in (2,\infty)$, $\ve>0$, and $\kappa\in (0,1)$ be fixed and note that, due to \eqref{33} and \eqref{curv}, steady-state solutions to \eqref{u}-\eqref{bcpsi} are characterized by
\bqn\label{ss}
\partial_x^2 u= {\lambda}\ { \frac{(1+\e^2(\partial_x u)^2)^{5/2}}{(1+u)^2}} \, \vert\partial_\eta\phi_u(\cdot,1)\vert^2=:\lambda\,h_\ve(u)\ ,\quad x\in I\ ,\qquad u(\pm1)=0\ ,
\eqn
the function $\phi_u$ being defined in \eqref{phiv}. Thus we may proceed as in \cite[Thm.~1.3~(i)]{ELW_CPAM}: First, it follows exactly as in \cite[Prop.~2.1]{ELW_CPAM} that the mapping $h_\ve:S_q(\kappa)\rightarrow L_q(I)$ is analytic. Thus, since $-A(0)\in\mathcal{L}(W_{q,D}^2(I),L_q(I))$ is invertible, the operator $A(\cdot)$ being defined in \eqref{opA}, we obtain that the mapping
$$
F:\R\times S_q(\kappa)\rightarrow W_{q,D}^2(I)\ ,\quad (\lambda,v)\longmapsto v + \lambda A(0)^{-1}h_\ve(v)
$$
is analytic with $F(0,0)=0$ and $D_vF(0,0) =\mathrm{id}_{W_{q,D}^2}$. Now, the Implicit Function Theorem ensures the existence of $\delta=\delta(\kappa)>0$ and an analytic function 
$$
[\lambda\mapsto U_\lambda]:[0,\delta)\rightarrow W_{q,D}^2(I)
$$ 
such that $F(\lambda,U_\lambda)=0$ for $\lambda\in [0,\delta)$. Hence $(U_\lambda,\Psi_\lambda)$ is the unique steady-state to
\eqref{u}-\eqref{bcpsi} satisfying $U_\lambda\in S_q(\kappa)$ and $\Psi_\lambda\in W_2^2(\Omega(U_\lambda))$ when $\lambda\in (0,\delta)$. To improve the regularity of $(U_\lambda,\Psi_\lambda)$ as stated in \eqref{new1}, we may argue as follows: Since \eqref{ss} is basically the same equation as considered in \cite[Thm.~1]{LaurencotWalker_ARMA}, a steady-state solution $(u,\psi)$ to \eqref{u}-\eqref{bcpsi} possessing the regularity property \eqref{new1} can be constructed by means of Schauder's fixed point theorem applied to \eqref{ss} with $u\in W_\infty^2(I)\cap S_q(\kappa)$ for each $\lambda$ sufficiently small. Hence, making $\delta>0$ smaller, if necessary,  uniqueness guarantees $(u,\psi)=(U_\lambda,\Psi_\lambda)$. This proves Theorem~\ref{TStable}~(i).  \\

To prove part~(ii) of Theorem~\ref{TStable}, we proceed similarly as in \cite[Thm.~1.3~(ii)]{ELW_CPAM} and use the Principle of Linearized Stability.  Let $\lambda\in (0,\delta)$ be given and write $v=u-U_\lambda$. Then, introducing $Q\in C^\infty(S_q(\kappa),L_q(I))$ with $Q(U_\lambda)=0$ by   $Q(u):=-A(\ve u)u-\lambda g_\ve(u)$, the linearization of \eqref{CP} reads
$$
\frac{\rd}{\rd t}v-D_uQ(U_\lambda)v=Q(v+U_\lambda)-D_uQ(U_\lambda)v=:G_\lambda(v)
$$
with $G_\lambda\in C^\infty(O_\lambda,L_q(I))$ being defined on some open zero neighborhood $O_\lambda$ in $W_{q,D}^2(I)$ such that $U_\lambda+O_\lambda\subset S_q(\kappa)$. In view of \eqref{curv} we obtain
\bqn\label{lin}
\frac{\rd}{\rd t}v +\big(A(\ve U_\lambda)+B_\lambda\big)v=G_\lambda(v)\ ,
\eqn
where
$$
B_\lambda v:=\lambda g_\ve(U_\lambda)\,\frac{3\,\ve^2\,\partial_xU_\lambda}{1+\ve^2(\partial_x U_\lambda)^2}\partial_x v+\lambda\, D_ug_\ve(U_\lambda) v\ .
$$
Since $U_\lambda\in S_q(\kappa)$, we have {  $A(\ve U_\lambda)\in \mathcal{H}(W_{q,D}^2(I),L_q(I);k,\omega)$ with spectral bound less than $-\omega<0$}. Thus, since $\|B_\lambda\|_{\mathcal{L}(W_{q,D}^2(I),L_q(I))}\rightarrow 0$ as $\lambda\rightarrow 0$, it follows from \cite[I.~Cor.~1.4.3]{LQPP} that $-(A(\ve U_\lambda)+B_\lambda)$ is the generator of an analytic semigroup on $L_q(I)$ and there is $\omega_1>0$ such that the complex half plane $[\mathrm{Re}\, z\ge -\omega_1]$ belongs to the resolvent set of $-(A(\ve U_\lambda)+B_\lambda)$ provided that $\lambda$ is sufficiently small. 
Now we may apply \cite[Thm.~9.1.2]{Lunardi} and conclude statement~(ii) of Theorem~\ref{TStable} by making $\delta>0$ smaller, if necessary.\\

Let us note that  also the corresponding potential converges exponentially toward the steady-state as $t\rightarrow\infty$.

\begin{rem}\label{RR}
In \cite{ELW_CPAM} (see (2.29) in the proof of Prop.2.1 therein) the following analogue of Lemma~\ref{L2a} 
$$
\|\phi_v-\phi_w\|_{W_2^{2}(\Omega)}\le c(\kappa,\ve)\, \|v-w\|_{W_{q}^{2}(I)}\ ,\quad v,w\in \overline{S}_q(\kappa)\ ,
$$
was shown to hold.
Consequently, under the assumptions of Theorem~\ref{TStable}~(ii) there is $R_1>0$ such that
$$
\|\phi_{u(t)}-\phi_{U_\lambda}\|_{W_2^2(\Omega)}\le R_1 e^{-\omega_0 t}\| u^0-U_\lambda\|_{W_{q,D}^2(I)}\ ,\quad t\ge 0\ ,
$$
which shows that also $\psi=\phi_{u(t)}\circ T_{u(t)}$ (with transformation $T_{u(t)}$ defined in \eqref{Tu}) converges exponentially to~$\Psi_\lambda$ as $t\rightarrow\infty$.
\end{rem}

\section{Non-Existence of Steady-State Solutions: Proof of Theorem~\ref{NoSS}}\label{SectNoSS}

Consider a steady-state solution $(u,\psi)$  to \eqref{u}-\eqref{bcpsi} with $u\in W_{q,D}^2(I)$, $\psi\in W_2^2(\Omega(u))$, and $u(x)>-1$ for $x\in I$. Recall that, according to \eqref{u}, \eqref{curv}, and the identity $$\partial_x\psi(x,u(x)) = - \partial_x u(x)\, \partial_z \psi(x,u(x))\ ,$$ the function $u$ solves
\begin{equation*}
\partial_x^2 u(x) =  \lambda\ \left( 1+\varepsilon^2\, (\partial_x u(x))^2)\right)^{5/2}|\partial_z\psi(x,u(x))|^2 \ ,\quad x\in I\  ,
\end{equation*}
with $u(\pm 1)=0$, so that $u$ is negative in $I$ by the comparison principle and convex. Thanks to the latter property, \cite[Lem.~4.1]{LaurencotWalker_ARMA} ensures that
$$
\partial_z\psi(x,u(x))\ge 1\ ,\quad x\in I\ .
$$
Hence, $u$ satisfies the following differential inequality
\begin{equation*}
\partial_x^2 u(x) \ge  \lambda\ \left( 1+\varepsilon^2\, (\partial_x u(x))^2)\right)^{5/2}\ ,\quad x\in I\  .
\end{equation*}
Introducing 
$$
J(r) := \int_0^r \frac{\rd s}{\left( 1 + s^2 \right)^{5/2}} = \frac{r(2r^2+3)}{3 (r^2+1)^{3/2}}\ , \quad r\ge 0\ , 
$$
the function $J$ maps $[0,\infty)$ onto $[0,2/3)$ and the previous differential inequality for $u$ reads
\begin{equation}
\partial_x J(\ve \partial_x u) \ge \lambda\,\ve\ , \qquad x\in I\ . \label{nex1}
\end{equation}
Let $x_m$ be a point of minimum of $u$. Since $u<0$ in $I$ and $x\mapsto u(-x)$ is also a steady-state, we may assume without loss of generality that $x_m\in (-1,0)$. Integrating \eqref{nex1} over $(x_m,x)$ for $x\in [0,1]$ gives
\begin{equation}
J(\ve \partial_x u(x)) \ge J(0) + \lambda\,\ve\, (x-x_m) \ge \lambda\,\ve\, x\ , \qquad x\in [0,1]\ . \label{nex2}
\end{equation}
Now, either $\lambda\,\ve \ge 2/3$ and we deduce from \eqref{nex2} that $J(\ve \partial_x u(1))\ge \lambda\,\ve \ge 2/3$ which contradicts the boundedness of $\partial_x u$. Or $\lambda\,\ve < 2/3$ and, since $J$ is concave, we infer from \eqref{nex2} after integration over $(0,1)$ and Jensen's inequality that
$$
J(-\ve\, u(0)) = J\left( \int_0^1 \ve\,\partial_x u(x)\ \rd x \right) \ge \int_0^1 J\left( \ve\,\partial_x u(x) \right)\ \rd x \ge \frac{\lambda\,\ve}{2}\ .
$$
If $\lambda\ge 2 J(\ve)/\ve$, the previous inequality and the monotonicity of $J$ entail that $-u(0)\ge 1$ and a contradiction again.   
Now, defining $\bar{\lambda}(\ve) := \min{\{ 2 J(\ve) , 2/3 \}}/\ve$ and noticing that $\bar{\lambda}(\ve)\rightarrow 2$ as $\ve\rightarrow 0$, Theorem~\ref{NoSS} follows.

\section{Small Aspect Ratio Limit: Proof of Theorem~\ref{Bq}}\label{Sec6} 

To prove Theorem~\ref{Bq} fix {$\lambda>0$, $q\in (2,\infty)$, and $u^0\in W_{q,D}^2(I)$ such that $-1<u^0\le 0$ in $I$. Clearly, there is $\kappa\in (0,1)$ such that $u^0\in S_q(\kappa)$. For $\varepsilon\in (0,1)$ let $(u_\varepsilon,\psi_\varepsilon)$ be the unique solution
to \eqref{u}-\eqref{bcpsi} which is defined on the maximal interval of existence $[0,T_m^\varepsilon)$. 
In the following, $(K_i)_{i\ge 1}$ denote positive constants depending only on $q$ and $\kappa$, but not on $\ve>0$ sufficiently small.

Set $\kappa_1 := \kappa/(2c_*(\kappa))< \kappa$, where $c_*(\kappa)\ge 1$ is the constant defined in \eqref{UA}. The continuity properties of $u_\varepsilon$ ensure
\begin{equation}
\tau^\varepsilon := \sup{\left\{ t\in [0,T_m^\varepsilon)\ :\ u_\varepsilon(s)\in\overline{S}_q(\kappa_1) \;\;\text{ for all }\;\; s\in [0,t] \right\}} > 0\ . \label{pam2}
\end{equation}
Owing to the continuous embedding of $W_q^2(I)$ in $W_\infty^1(I)$, there is a positive constant $K_1$ such that, for all $\varepsilon>0$,
\begin{align}
-1 + \kappa_1 \le u_\varepsilon(t,x) & \le 0\,, \qquad \ (t,x)\in [0,\tau^\varepsilon]\times [-1,1]\,, \label{z1} \\
\|u_\varepsilon(t)\|_{W_q^2(I)} + \|u_\varepsilon(t)\|_{W_\infty^1(I)} & \le K_1\,, \qquad t\in [0,\tau^\varepsilon]\,. \label{z2}
\end{align}
As a consequence of \eqref{z2} there is $\varepsilon_0>0$ depending only $q$ and $\kappa$} such that 
\begin{equation}
\varepsilon_0^2\ \left\| \partial_x u_\varepsilon(t) \right\|_{L_\infty(I)}^2 \le \frac{1}{2}\,, \qquad (t,\varepsilon)\in [0,\tau^\varepsilon]\times (0,\varepsilon_0]\,. \label{z3}
\end{equation}
For $\e\in (0,\varepsilon_0)$, we set 
$$
\phi_\e(t) := \phi_{u_\varepsilon(t)}=\psi_\varepsilon(t) \circ T_{\ue(t)}^{-1}\ ,\qquad t\in [0,\tau^\varepsilon]\ ,
$$ 
with $T_{\ue(t)}^{-1}$ given by \eqref{Tuu} and 
$$
\Phi_\varepsilon(t,x,\eta) := \phi_\varepsilon(t,x,\eta)-\eta\ ,\qquad (t,x,\eta)\in [0,\tau^\varepsilon]\times\overline{\Omega}\,.
$$ 
We first recall uniform estimates on  $\Phi_\varepsilon$ that have been established in \cite[Lem.~5.1]{ELW_CPAM}:

\begin{lem}\label{le:z1}
There exists a positive constant $K_2$ such that, for $\varepsilon\in (0,\varepsilon_0)$ and $t\in [0,\tau^\varepsilon]$, 
\begin{align}
\left\| \partial_x\Phie(t) \right\|_{L_2(\Omega)} + \frac{1}{\e}\ \left( \left\| \Phie(t) \right\|_{L_2(\Omega)} + \left\| \partial_\eta\Phie(t) \right\|_{L_2(\Omega)} \right) & \le K_2\,, \label{z41} \\
\frac{1}{\varepsilon} \left\| \partial_x\partial_\eta\Phie(t) \right\|_{L_2(\Omega)} + \frac{1}{\varepsilon^2}\ \left\| \partial_\eta^2\Phie(t)\right\|_{L_2(\Omega)} & \le K_2\,, \label{z42} \\
\frac{1}{\varepsilon}\ \left\| \partial_\eta{ \Phie(t,\cdot,1)} \right\|_{W_2^{1/2}(I)} & \le K_2\,. \label{z5}
\end{align}
\end{lem}
At this point let us mention that the assumption  $u^0\le 0$ is used to obtain the previous lemma.
We then deduce from Lemma~\ref{le:z1} a positive lower bound on $\tau^\ve$. 

\begin{lem}\label{le.pam1}
\begin{itemize}
\item[(i)] There is $\tau>0$ depending only on $q$, $\lambda$, and $\kappa$ such that $\tau^\varepsilon\ge \tau$ for all $\varepsilon\in (0,\varepsilon_0)$.
\item[(ii)] There is $\Lambda:=\Lambda(\kappa)>0$ such that $\tau^\varepsilon=T_m^\ve=\infty$ for all $\ve\in (0,\ve_0)$ provided $\lambda\in (0,\Lambda)$.
\end{itemize}
\end{lem}

\begin{proof}[{\bf Proof}]
Owing to \eqref{z5} we have $\|\partial_\eta\phi_\ve(t,\cdot,1)\|_{W_{2}^{1/2}(I)}\le 1+ K_2 \ve$ while \eqref{z1} and \eqref{z2}  imply that
$$
\left\|\frac{1+\e^2(\partial_x u_\ve(t))^2}{(1+u_\ve(t))^2}\right\|_{W_q^1(I)}\le K_3
$$
for $t\in [0,\tau^\ve]$. Consequently, from the continuity of pointwise multiplication 
$$
W_q^1(I)\cdot W_2^{1/2}(I)\cdot W_2^{1/2}(I) \hookrightarrow W_2^{2\sigma}(I)
$$
for $2\sigma\in (0,1/2)$ fixed (see Theorem~\ref{PM} below), we conclude that there is $K_4>0$ such that
\begin{equation}
\left\| g_\varepsilon(u_\e(t)) \right\|_{W_2^{2\sigma}(I)} \le K_4\ , \qquad t\in [0,\tau^\varepsilon]\ . \label{pam3}
\end{equation}
As in the proof of \eqref{zui} and \eqref{11}, we infer from \eqref{UA}, \eqref{pam3}, the fact that $u^0\in S_q(\kappa)$, and the Variation-of-Constant formula that, for $t\in [0,\tau^\ve]$, 
\begin{align}
\|u_\varepsilon(t)\|_{W_{q,D}^{2}(I)} & \le  c_*(\kappa)\ \|u^0\|_{W_{q,D}^2(I)} + \lambda\ c_*(\kappa)\ \int_0^t e^{-\vartheta(t-s)}\ (t-s)^{\sigma-1-\frac{1}{2}(\frac{1}{2}-\frac{1}{q})}\ \|g_\varepsilon(u_\varepsilon(s))\|_{W_{2,D}^{2\sigma}(I)}\,\rd s \nonumber \\
&\le  \frac{c_*(\kappa)}{\kappa}+\lambda\ c_*(\kappa)\ K_4\ \int_0^t e^{-\vartheta s}\ s^{\sigma-1-\frac{1}{2}(\frac{1}{2}-\frac{1}{q})}\, \rd s\ , \label{pam4}
\end{align}
and, using in addition the embedding of $W_{q,D}^2(I)$ in $L_\infty(I)$ with constant 2,
\begin{align}
u_\varepsilon(t) \ge & -1 + \kappa - 2\,\lambda\,  \int_0^t \left\| U_{A(\ve u_\ve)}(t,s)\, g_\varepsilon(u_\varepsilon(s)) \right\|_{W_{q,D}^2(I)}\, \rd s \nonumber \\
\ge & -1+\kappa -  2\,\lambda\,  c_*(\kappa) \,  \int_0^t e^{-\vartheta (t-s)}\ (t-s)^{\sigma-1-\frac{1}{2}(\frac{1}{2}-\frac{1}{q})}\ \|g_\varepsilon(u_\varepsilon(s))\|_{W_{2,D}^{2\sigma}(I)}\,\rd s \nonumber\\
\ge & -1+\kappa -  2\,\lambda\, c_*(\kappa)\, K_4\ \int_0^t e^{-\vartheta s}\ s^{\sigma-1-\frac{1}{2}(\frac{1}{2}-\frac{1}{q})}\, \rd s\ . \label{pam5}
\end{align}
Since the integral above converges to zero as $t\to 0$, there exists $\tau>0$ which depends only on $q$ and $\kappa$ such that
\bqn\label{zui2}
\int_0^t e^{-\vartheta s}\ s^{\sigma-1-\frac{1}{2}(\frac{1}{2}-\frac{1}{q})}\, \rd s < \min\left\{\frac{1}{\lambda \kappa K_4} \,,\,\frac{(2c_*(\kappa)-1)\kappa}{4\lambda c_*(\kappa)^2 K_4}\right\} \;\;\text{ for all }\;\; t\in [0,\tau]\ .
\eqn
Thanks to this choice, we readily deduce from \eqref{pam4} and \eqref{pam5} that $\|u_\varepsilon(t)\|_{W_{q,D}^{2}(I)}  \le 1/ \kappa_1$ and $u_\varepsilon(t)\ge -1+  \kappa_1$ for all $t\in [0,\tau]\cap [0,\tau^\varepsilon]$. Therefore, $u_\varepsilon(t)\in\overline{S}_q( \kappa_1)$ for all $t\in [0,\tau]\cap [0,\tau^\varepsilon]$ and the definition of $\tau^\varepsilon$ implies that $\tau^\varepsilon\ge \tau$. Finally, it is obvious from \eqref{zui2} that there is $\Lambda(\kappa)>0$ such that \eqref{zui2} holds for any $\tau>0$ and $\lambda\in (0,\Lambda(\kappa))$. This implies that $\tau^\ve\ge \tau$ for any $\tau>0$ so that $\tau^\ve =\infty$.
\end{proof}

Based on these auxiliary Lemmas~\ref{le:z1} and \ref{le.pam1} we may proceed exactly as in the proof of \cite[Thm.1.4]{ELW_CPAM} to conclude the statement of Theorem~\ref{Bq}. In particular, an obvious consequence of Lemma~\ref{le:z1} and Lemma~\ref{le.pam1} is that $\phi_\ve\rightarrow \eta$ in suitable topologies as $\ve\rightarrow 0$ and that $(u_\ve)_{\ve\in (0,\ve_0)}$ is bounded in $C([0,\tau],W_{q,D}^2(I))$.

\section{Appendix}

In this appendix we recall a useful tool on pointwise multiplication of functions in Sobolev spaces which is used frequently in the previous sections.

\begin{thm}\label{PM}
Let $\Om$ be an open and {non-empty} subset of $\R^{n}$ with finite volume. Let $m\ge 2$ be an integer and let
$p,p_{j}\in [1,\infty)$ and $s,s_{j}¥\in (0,\infty )$ for $1\le j\le m$ be real numbers satisfying $s \leq \min \{s_{j}\}$ and
$$
s-\frac{n}{p} < \left\{ 
\begin{array}{ll}
\displaystyle{\sum\limits_{s_{j}<\frac{n}{p_{j}}}(s_{j}-\frac{n}{p_{j}})} &
\text{if} \quad \displaystyle{\min_{1\le j\le m}\left\{ s_{j}-\frac{n}{p_{j}} \right\} < 0} \ , \\
\displaystyle{\min_{1\le j \le m}\left\{ s_{j}-\frac{n}{p_{j}} \right\}} & \text{otherwise}\ .
\end{array}\right.
$$
Then pointwise multiplication
\begin{equation*}
\prod^{m}_{j=1}W_{p_{j}}^{s_{j}}(\Om) \rightarrow
W_{p}^{s}(\Om )
\end{equation*}
is continuous.
\end{thm}

Theorem~\ref{PM} is a consequence of the more general result stated in \cite[Thm.~4.1]{AmannMultiplication} (see also Remark~4.2~(d) therein). It follows by observing that the Sobolev spaces $W_p^s(\Omega)$ coincide with the Besov spaces $B_{p,p}^s(\Omega)$ provided $s\in (0,\infty)\setminus\N$ and $p\in [1,\infty)$ together with the fact that $W_p^{s_1}(\Omega)\hookrightarrow B_{p,p}^{s_2}(\Omega)$ if $s_1>s_2>0$.


\section*{Acknowledgments}

Part of this research was done while J.E. and Ch.W. were visiting the Institut de Math\'{e}ma\-ti\-ques de Toulouse, Universit\'{e} Paul Sabatier. The financial support and kind hospitality is gratefully acknowledged. The work of Ph.L. was partially supported by the Centre International de Math\'ematiques et d'Informatique CIMI. We thank Sylvie Monniaux for pointing out to us the paper \cite{JerisonKenig}.




\end{document}